\documentclass{article}[11pt]
 \usepackage{amsmath,amsfonts,mathscinet,graphicx}
\usepackage{amssymb, amsthm}
\usepackage{yfonts}

\newtheorem{theorem}{Theorem}[section]
\newtheorem{lemma}[theorem]{Lemma}
\newtheorem{proposition}[theorem]{Proposition}
\newtheorem{corollary}[theorem]{Corollary}

\newtheorem{remark}[theorem]{Remark}
\newcommand{\tn}{|\mspace{-1mu}|\mspace{-1mu}|}

\numberwithin{equation}{section}

\begin{document}

% \title[short text for running head]{full title}
\title{A stabilized nonconforming finite element method for the elliptic Cauchy problem}

%    Only \author and \address are required; other information is
%    optional.  Remove any unused author tags.

%    author one information
% \author[short version for running head]{name for top of paper}

\author{Erik Burman \\Department of Mathematics \\University College
  London}
%\curraddr{xxx}
%\email{e.burman@ucl.ac.uk}
%\thanks{xxx}

%    author two information
%\author{}
%\address{}
%\curraddr{}
%\email{}
%\thanks{}

%    \subjclass is required.

%    Abstract is required.

\maketitle

\begin{abstract}
In this paper we propose a nonconforming finite element method for the
solution of the ill-posed elliptic Cauchy problem. We prove
error estimates using continuous dependence estimates in the
$L^2$-norm. The effect of
perturbations in data on the estimates is investigated.
The recently derived framework from \cite{Bu13,Bu14} is extended to
include the
case of nonconforming approximation spaces and we show that the use of
such spaces allows us to reduce the amount of stabilization necessary
for convergence, even in the case of ill-posed problems. 
\end{abstract}
 
%    Text of article.
\section{Introduction}\label{sec:intro}
We consider the Cauchy problem for Poisson's equation in a bounded
domain. This problem is known to be severely
ill-posed in the sense of Hadamard \cite{Hada02,Belg07,ARRV09}. The ill-posedness makes numerical
approximation challenging and different regularization methods have
been proposed, such as Tikhonov regularization \cite{TA77} or the quasi
reversibility method introduced by Latt\`es and Lions \cite{LL69}. 

Various finite element approaches for the solution of the elliptic
Cauchy problem have been suggested in the litterature. Some are based
on standard Galerkin formulations, but rely on structured
meshes or a special form of the continuous problem for stability
\cite{FM86,Lu95,RHD99}. Some use the above mentioned regularization techniques to
ensure stability \cite{ABB06,ABF06,Bour05, Bour06, DHH13} a related
approach is to recast the problem as a minimization problem
\cite{CN06,HHKC07,HLT11},
possibly with regularization.

 The objective of the present work is to draw on the ideas of
 \cite{Bu13,Bu14} and propose a consistent stabilization of a
 non-conforming finite element method.
The upshot is that the use of nonconforming elements
allows us to use a standard stabilization operators known from previous works on
well-posed problems \cite{HL02, HL03, BH05} for stability. Indeed we
only need to apply a penalty on the jump of the
approximate solution over element faces. The structure of the method bears some ressemblance
to that introduced in \cite{Bour05}, but the stabilizing terms in our
case are consistent for exact solutions in $H^1(\Omega)$. The key observation here is that the functional to be
minimized has on effect only on the discrete space, indeed it is zero
for any function in $H^1$.  The associated Euler-Lagrange equations
result in a consistent stabilized finite element method.

The fact that the stabilization is consistent allows us to derive
error estimates using discrete stability and the continuous dependence
on data of the partial differential equation. We follow an approach
similar to that suggested in \cite{Bu14}, but in this case an inf-sup
condition is necessary for the discrete stability. The error bound
is on a posteriori form, using a residual quantity together with the
continuous dependence. Thanks to the primal/adjoint stabilization the
residual terms can be shown to be optimally convergent independent of
the stability of the underlying problem, for sufficiently smooth
solutions. 

We also show how perturbed data can be introduced in the analysis and discuss how a
posteriori control of the mesh refinement may include the effect of
perturbations, provided their magnitude is known.

The problem that we are interested in takes the form:
find $u:\Omega\mapsto \mathbb{R}$ such that
\begin{equation}
\label{eq:Pb}
\left\{
\begin{array}{rcl}
-\Delta  u & = &  f, \mbox{ in } \Omega\\
u  &=&  0 \mbox{ on } \Gamma_D\\
\nabla u \cdot n &=&\psi \mbox{ on } \Gamma_C
\end{array}
\right. 
\end{equation}
where $\Omega \subset \mathbb{R}^d$, $d=2,3$ is a convex polyhedral (polygonal)
domain and $\Gamma_C,\, \Gamma_D$ denote simply connected parts of the
boundary $\partial \Omega$, such that $\Gamma_C \subset \Gamma_D \subset \partial \Omega$. We denote the
complement to the Dirichlet boundary $\Gamma_D' := \partial \Omega
\setminus \Gamma_D$ and the complement of the Cauchy boundary
$\Gamma_C' := \partial \Omega
\setminus \Gamma_C$. For simplicity we will assume that both
$\Gamma_C$, $\Gamma_D$ and $\Gamma_D'$ have strictly positive
$(d-1)$-measure. The practical interest of \eqref{eq:Pb} stems from
engineering problems where the exact
boundary condition is unknown on part of the boundary, but additional
measurements $\psi$ of the fluxes are available on the accessible
boundary $\Gamma_D$. It is then reasonable to assume that there exists
a unique solution with a certain regularity and then prove convergence
of the numerical method under these assumptions. This is the approach
we will take below. To this end we assume that $f \in L^2(\Omega)$, $\psi \in H^{\frac12}(\Gamma)$ and
that a unique $u \in H^2(\Omega)$ satisfies \eqref{eq:Pb}.

For the derivation of a weak formulation we introduce the spaces
$
V:= \{v \in H^1(\Omega): v\vert_{\Gamma_D} = 0 \}$ and
$W:= \{v \in H^1(\Omega): v\vert_{\Gamma_C'} = 0 \},
$ both equipped with the $H^1$-norm and with dual spaces denoted by $V'$
and $W'$.

Using these spaces we obtain a weak formulation: find $u \in V$ such that
\begin{equation}\label{abstract_prob}
a(u,w) = l(w)\quad \forall w \in W,
\end{equation}
where
$$a(u,w) = \int_\Omega \nabla u\cdot \nabla w ~\mbox{d}x,$$ and $$l(w) := 
\int_\Omega f w ~\mbox{d}x + \int_{\Gamma_C} \psi \, w ~\mbox{d}s.
$$
We will use the notation $(\cdot,\cdot)_{X}$ for the $L^2$-scalar product over
$X$ and for the associated norm we write $\|x\|_X :=
(x,x)_{X}^{\frac12}$. The $H^s$-norm will be denoted by
$\|\cdot\|_{H^s(\Omega)}$ and we identify the norms on $V$ and $W$
with the $H^1$-norm, $\|\cdot \|_V = \|\cdot \|_W =
\|\cdot\|_{H^1(\Omega)}$. Observe that we may not assume that the
problem is well-posed for general $l(\cdot) \in W'$. Indeed since $u
\not \in W$ coercivity fails and inf-sup stability does not hold
either in general \cite{Belg07}. 
\subsection{Continuous dependence on data}\label{continuous_dep}
The problem \eqref{eq:Pb} is ill-posed and for our analysis we will only use a
continuous dependence result linking the size of some functional to
the solution to the size of data. Consider the functional $j:V \mapsto
\mathbb{R}$. Let $\Xi:\mathbb{R}^+\mapsto \mathbb{R}^+$ be a continuous, monotone increasing
function with $\lim_{x \rightarrow 0^+}\Xi(x) = 0$. The continuous
dependence that we will assume then takes the form.
\begin{equation}\label{cont_dep_assump}
\mbox{If for $\epsilon > 0$, there holds
$
\|l\|_{W'} \leq \epsilon 
$
in \eqref{abstract_prob} then, for $\epsilon$ small enough,
$
|j(u)| \leq \Xi(\epsilon).
$}
\end{equation}
It is known \cite[Theorems 1.7 and 1.9]{ARRV09} that if there exists a solution $u \in
H^1(\Omega)$, with $E:=\|u\|_{H^1(\Omega)}$ to \eqref{eq:Pb}, a continuous dependence of the form \eqref{cont_dep_assump}
holds for $0<\epsilon<1$ and 
\begin{equation}\label{stab_1}
\begin{array}{l}
\mbox{$j(u):=
\|u\|_{L^2(\omega)}$, $\omega \subset \Omega:\,
\mbox{dist}(\omega, \partial \Omega) =: d_{\omega,\partial \Omega}>0$}\\[3mm]
\mbox{
with $\Xi(x) = C(E) x^\varsigma$, $C(E)>0$, $\varsigma:=
\varsigma(d_{\omega,\partial \Omega}) \in (0,1)$}
\end{array}
\end{equation}
 and for \begin{equation}\label{stab_2}
\begin{array}{l}
\mbox{$j(u):= \|u\|_{L^2(\Omega)}$ with $\Xi(x) = C_1(E)
(|\log(x)| + C_2(E))^{-\varsigma}$}\\[3mm]
\mbox{ with $C_1(E), C_2(E)>0$, $\varsigma \in
(0,1)$.}
\end{array}
\end{equation}
The constants above also depend on the geometry of the
problem. Note that to derive these results $l(\cdot)$ is first associated
with its Riesz representant in $W$ (c.f. \cite[equation
(1.31)]{ARRV09} and discussion.)

\section{The nonconforming stabilized method}\label{sec_quasi}
Let $\{\mathcal{T}_h\}_h$ denote a family of shape regular and quasi uniform
tesselations of $\Omega$ into nonoverlapping simplices, such that for
any two different simplices $\kappa$, $\kappa' \in \mathcal{T}_h$, $\kappa \cap
\kappa'$ consists of either the empty set, a common face or a common
vertex. The diameter of a simplex $\kappa$ will be denoted
$h_{\kappa}$ and the outward pointing normal $n_{\kappa}$. The family
$\{\mathcal{T}_h\}_h$ is indexed by the maximum element-size of
$\mathcal{T}_h$, $h :=
\max_{\kappa \in \mathcal{T}_h} h_\kappa$. We denote the set of element faces in $\mathcal{T}_h$ by
$\mathcal{F}$ and let $\mathcal{F}_i$ denote the set of interior faces and $\mathcal{F}_{\Gamma}$ the set of faces in some
$\Gamma \subset \partial \Omega$. To each face $F \in \mathcal{F}$ we
associate the mesh parameter $h_F := \mbox{diam}(F)$. We will assume that the mesh is
fitted to the subsets of $\partial \Omega$ representing the boundary
conditions $\Gamma_D$ and $\Gamma_C$, so that the
boundaries of these subsets coincide with element boundaries. 
To each face
$F$ we associate a unit normal vector, $n_F$. For interior faces its orientation is arbitrary, but fixed. On the boundary
$\partial \Omega$ we identify $n_F$ with the outward pointing normal
of $\Omega$.
We define the jump over interior faces $F \in
\mathcal{F}_i$ by $[v]\vert_F:= \lim_{\epsilon \rightarrow
  0^+} (v(x\vert_F- \epsilon n_F) - v(x\vert_F+ \epsilon n_F))$
and for faces on the boundary, $F \in \partial \Omega$, we let
$[v]\vert_F := v \vert_F$. Similarly we define the average of a function over
an interior
face $F$ by $\{v \}\vert_F := \tfrac12 \lim_{\epsilon \rightarrow
  0^+} (v(x\vert_F- \epsilon n_F) + v(x\vert_F+ \epsilon n_F))$ and
for $F$ on the boundary we define $\{v \}\vert_F := v \vert_F$.
The classical nonconforming space of piecewise affine
finite element functions (see \cite{CR73}) then reads
$$
X_h^{\Gamma} := \{v_h \in L^2(\Omega): \int_{F} [v_h]~
\mbox{d}s = 0,\, \forall F \in \mathcal{F}_i \cup \mathcal{F}_{\Gamma}
\mbox{ and } v_h\vert_{\kappa} \in \mathbb{P}_1(\kappa),\,
\forall \kappa \in \mathcal{T}_h \}
$$
where $\mathbb{P}_1(\kappa)$ denotes the set of polynomials of degree less than
or equal to one restricted to the element $\kappa$ and $\Gamma$
denotes some portion of the boundary $\partial \Omega$ consisting of a
union of a subset of boundary element faces.
We may then define the spaces $V_h := X_h^{\Gamma_D}$ and $W_h :=
X_h^{\Gamma_C'}$. We recall the interpolation operator $r_h: H^1(\Omega)
\rightarrow X_h^{\Gamma}$ defined by the relation 
\[
\overline{\{r_h v\}}\vert_F := |F|^{-1} \int_F \{r_h v\}~\mbox{d}s
=  |F|^{-1}  \int_F  u ~\mbox{d}s 
\]
for every $F \in \mathcal{F}$ and with $|F|$ denoting the $(d-1)$-measure
of $F$.
It is conventient to introduce the broken
norms 
\[
\|x\|_h^2 := \sum_{\kappa \in \mathcal{T}_h} \|x\|_{\kappa}^2 \mbox{
  and } \|x\|_{1,h}^2 := \|x\|_h^2+ \|\nabla x\|_h^2.
\]
The following inverse and trace inequalities are well known
\begin{equation}\label{trace}
\begin{array}{rcl}
\|v\|_{\partial \kappa} \leq C_t (h_{\kappa}^{-\frac12} \|v\|_{\kappa} +
h_{\kappa}^{\frac12} \|\nabla v\|_{\kappa}), \forall v \in H^1(\kappa) \\[3mm]
h_{\kappa}\|\nabla v_h\|_{\kappa} + h_\kappa^{\frac12} \|v_h\|_{\partial \kappa}
\leq C_i \|v_h\|_\kappa, \quad \forall v_h \in X_h^\Gamma.
\end{array}
\end{equation}
Using the inequalities of \eqref{trace} and standard approximation
results from \cite{CR73} it is straightforward to show the following
approximation results of the interpolant $r_h$
\begin{equation}\label{eq:approx}
\begin{array}{rcl}
\|u - r_h u\|    +h \|\nabla (u - r_h u)\|_h &\leq &C h^{t} |u|_{H^t(\Omega)}\\[3mm]
 \|h^{-\frac12} ( u - r_h
u) \|_{\mathcal{F}} + \|h^{\frac12} \nabla (u - r_h u) \cdot n_F
\|_{\mathcal{F}} &\leq &C h^{t-1} |u|_{H^t(\Omega)}
\end{array}
\end{equation}
where $t \in \{1,2\}$. It will also be useful to bound the $L^2$-norm
of the interpolant $r_h$ by its values on the element faces. To this
end we prove a technical lemma.
\begin{lemma}\label{lem:invtrace}
For any function $v_h \in X_h^\Gamma$ there holds
\[
\|h^{-1} v_h\|_\Omega \leq c_{\mathcal{T}}\left( \sum_{F \in \mathcal{F}} h^{-1}_F \|\overline{
  \{v_h\}}\|^2_F \right)^{\frac12}
\]
\end{lemma}
\begin{proof}
It follows by norm equivalence of discrete
spaces on the reference element and a scaling argument (under the
assumption of shape regularity) that for all
$\kappa \in \mathcal{T}_h$
\begin{equation}\label{eq:inv_trace}
%\|\zeta_h\|^2_{\kappa} \leq C \sum_{F \in \partial \kappa}
%h^2_F \|h_F^{\frac12}  [\nabla y_h \cdot n_F]\|_F^2. %\mbox{ and }followed
 \|v_h\|^2_{\kappa} \leq C \sum_{\substack{F \in \partial \kappa }}
%\\
%     F\not \in\mathcal{F}_{\Gamma}}}
h_F \|\overline{
  v}_h\|_F^2.
\end{equation}
The claim follows by shape regularity and by summing over the elements of $\mathcal{T}_h$ and
recalling that $\|\overline{v}_h\|_F^2 = \|\overline{\{v_h\}}\|_F^2 $.
\end{proof}
Following \cite{Bour05, Bu13} the formulation may now be written:
 find $(u_h,z_h) \in V_h \times W_h$ such that,
\begin{equation}\label{stabFEM}
\begin{array}{rcl}
a_h(u_h,w_h) - s_W(z_h,w_h) &=&l(w_h)\\[3mm]
a_h(v_h,z_h) + s_V(u_h,v_h)  &=&0%\lsproofabel{dual_eq}
\end{array} 
\end{equation}
for all $(v_h,w_h) \in V_h \times W_h$.
Here the bilinear forms are defined by
\[
a_h(u_h,w_h) = \sum_{\kappa \in \mathcal{T}_h} \int_\kappa \nabla u_h\cdot \nabla w_h ~\mbox{d}x,
\]
\begin{equation}\label{stab_LS}
s_W(z_h,w_h)  := \sum_{\kappa \in \mathcal{T}_h} \int_\kappa \gamma_W
\nabla z_h\cdot \nabla w_h~\mbox{d}x
\end{equation}
or 
\begin{equation}\label{eq:stab_dual_consist}
s_W(z_h,w_h)  := \sum_{F \in \mathcal{F}_i \cup \mathcal{F}_{\Gamma_C'}} \int_{F} \gamma_W
h_F^{-1} [z_h][w_h] ~\mbox{d}s
\end{equation}
and finally
\begin{equation}\label{stab_primal}
s_V(u_h,v_h)  := \sum_{F \in \mathcal{F}_i \cup \mathcal{F}_{\Gamma_D}} \int_{F} \gamma_V
h_F^{-1} [u_h][v_h] ~\mbox{d}s.
\end{equation}
We also propose the compact form: find $(u_h,z_h) \in \mathcal{V}_h :=
V_h \times W_h$ such that,
\[
A_h[(u_h,z_h),(v_h,w_h)] = l(w_h)
\]
for all $(v_h,w_h) \in \mathcal{V}_h$. The bilinear form is then
given by
\[
A_h[(u_h,z_h),(v_h,w_h)] := a_h(u_h,w_h) - s_W(z_h,w_h) + a_h(v_h,z_h) + s_V(u_h,v_h) .
\]
Observe that for \eqref{stab_LS}, by Poincar\'e's inequality there exists $c_1,c_2>0$ so that 
\[
c_1 \gamma_W^{\frac12}  \|w_h\|_{1,h} \leq  s_W(w_h,w_h)^{\frac12}
\leq c_2 \gamma_W^{\frac12} \|w_h\|_{1,h}, \forall w_h \in W_h.
\]
This norm equivalence is important for stability when there are
perturbations in data (see Lemma
\ref{lem:pert_bound}). For the weaker adjoint stabilization
\eqref{eq:stab_dual_consist} only the upper bound holds. For the first part of the analysis (sections
\ref{sec:stability}-\ref{sec:error}) the stability obtained by \eqref{eq:stab_dual_consist}
is sufficient and the analysis is identical. In Section
\ref{sec:perturbed_data} where perturbed data are considered the two
approaches lead to slightly different estimates. This operator has the advantage of being adjoint
consistent, but since duality arguments are not used herein this has
no impact on the results presented below. The stabilization
\eqref{stab_LS} will be considered in the analysis, but we will
outline in remarks how the arguments change if
\eqref{eq:stab_dual_consist} is used. We will then compare the behavior
of the two operators numerically.

We end this section by proving two technical Lemmas that will be
useful in the analysis.
Using the regularity assumptions on the data in $l(w)$
it is straightforward to show that the formulation satisfies the
following weak consistency
\begin{lemma}\label{lem:weakcons}(Weak consistency)
Let $u$ be the solution of \eqref{eq:Pb}, with $f \in L^2(\Omega)$ and
$\psi \in L^2(\Gamma_{C})$ and let $(u_h,z_h) \in \mathcal{V}_h$ be
the solution of \eqref{stabFEM} then,  for all $w_h \in W_h$, there holds,
\begin{equation}\label{galortho}
|a_h(u_h - u,w_h) - s_W(z_h,w_h)|
\leq \sum_{F \in \mathcal{F}_i \cup  \mathcal{F}_{\Gamma_C'}} \inf_{\nu_h \in V_h}\int_{F}
|(\nabla u - \{\nabla \nu_h\})\cdot n_F || [w_h]| ~\mbox{d}s.
\end{equation}
\end{lemma}
\begin{proof}
Multiplying \eqref{eq:Pb} with $w_h \in W_h$ and integrating by parts we have
\begin{multline}
\int_\Omega f w_h ~\mbox{d}x =-\int_\Omega \Delta u w_h ~\mbox{d}x \\
= -\sum_{\kappa \in \mathcal{T}_h}
\sum_{\substack{F \in \partial \kappa \\ F \not \in \Gamma_C}} \int_F \nabla u \cdot n_{\kappa} w_h ~
\mbox{d}s + a_h(u,w_h) - \int_{\Gamma_C} \psi w_h ~\mbox{d}s
\end{multline}
or by rearranging terms
\[
a_h(u,w_h) = l(w_h) + \sum_{\kappa \in \mathcal{T}_h}
\sum_{\substack{F \in \partial \kappa \\ F \not \in \Gamma_C}} \int_F \nabla u \cdot n_{\kappa} w_h ~
\mbox{d}s.
\]
Using \eqref{stabFEM} we obtain
\[
a_h(u_h - u,w_h) - s(z_h,w_h)
= -\sum_{\kappa \in \mathcal{T}_h}
\sum_{\substack{F \in \partial \kappa \\ F \not \in \Gamma_C}} \int_F \nabla u \cdot n_{\kappa} w_h ~
\mbox{d}s.
\]
By the definition of the finite element space $W_h$ on $\Gamma_C'$ and since every internal face appears twice with different orientation of
$n_{\kappa}$ we have for all $\nu_h \in V_h$,
\[
\sum_{\substack{F \in \partial \kappa \\ F \not \in \Gamma_C}} \int_F \nabla u \cdot n_{\kappa} w_h ~
\mbox{d}s = \sum_{\substack{F \in \partial \kappa \\ F \not \in
    \Gamma_C}} \int_F(\nabla u - \{\nabla \nu_h\})\cdot n_\kappa~  w_h ~
\mbox{d}s. 
\]
We now observe that by replacing $w_h$ with the jump $[w_h]$ we may write the sum over the
 faces of the mesh, replacing $n_{\kappa}$ by $n_F$. The conclusion follows by
 taking absolute values on both sides and moving the absolute values
 under the intergral sign creating the desired inequality.
\end{proof}
\begin{lemma}\label{lem:rhH1proj}
For any $v \in H^1(\Omega)$ and for all $w_h \in W_h$ there holds
\[
a(v - r_h v, w_h) = 0.
\]
\end{lemma}
\begin{proof}
By integration by parts we have
\[
a(v - r_h v, w_h) = \sum_{\kappa \in \mathcal{T}_h} \sum_{F
  \in \partial \kappa} \int_F (v - r_h v) \nabla w_h \cdot n_{\kappa}
~\mbox{d}s = 0,
\]
where the last equality is a consequence of the definition of $r_h u$.
\end{proof}
\section{Stability estimates}\label{sec:stability}
The issue of stability of the discrete formulation is crucial since we
have no coercivity or inf-sup stability of the continuous formulation \eqref{abstract_prob} to rely on. 
By taking $v_h = u_h$ and $w_h = -z_h$, and defining the semi-norm 
$
|v_h|_{s_V} := s_V(v_h,v_h)^{\frac12}, \, \forall v_h \in V_h$ 
  and the norm $\|w_h\|_{s_W} := s_W(w_h,w_h)^{\frac12}, \, \forall w_h \in W_h$
 we obtain the stability
estimate
\begin{equation}\label{basic_stability}
|u_h|^2_{s_V}+\|z_h\|^2_{s_W}=-l(z_h)
\end{equation}
showing that we have control of $z_h$  and of the nonconforming part of the
approximation of $u_h$. If the stabilization operator
\eqref{eq:stab_dual_consist} is used, $\|\cdot\|_{s_W}$ is a semi-norm
similar to $|\cdot|_{s_V}$. The stability \eqref{basic_stability} is of course insufficient for any
useful analysis, however we will use it here as a starting point for
an inf-sup argument that implies existence of a unique discrete solution. To this end we introduce a
mesh-dependent norm
\begin{equation}\label{def:tnorm1}
\tn v_h \tn_V := \gamma^{\frac12}_V \|h \nabla v_h\|_h + \gamma^{\frac12}_V\|h
[n_F \cdot\nabla v_h ]\|_{\mathcal{F}_i \cup \mathcal{F}_{\Gamma_C}}+ |v_h|_{s_V},
\end{equation}
where
\[
\|h^{\frac12} [n \cdot\nabla v_h ]\|_{\mathcal{F}_i \cup \mathcal{F}_{\Gamma_C}}^2 := \sum_{F \in
  {\mathcal{F}_i \cup \mathcal{F}_{\Gamma_C}}} h_F \|
[n_F \cdot\nabla v_h ]\|^2_F.
\]
The following approximation estimate is an immediate consequence of \eqref{eq:approx},
\begin{equation}\label{tnorm_approx}
\tn v - r_h v \tn_V \leq C \gamma_V^{\frac12} h
|v|_{H^2(\Omega)},\quad \forall v \in H^2(\Omega).
\end{equation}
We will also use the composite norm
\[
\tn (u_h,z_h) \tn := \tn u_h \tn_V +  \| z_h \|_{s_W}.
\]
Since Dirichlet boundary conditions are set weakly on $\Gamma_C$ in
$V_h$ and on 
$\Gamma_C'$ in $W_h$, $\tn (u_h,z_h) \tn$ is a norm, when \eqref{stab_LS} is
used. When \eqref{eq:stab_dual_consist}, the jump of $\nabla z_h$ and
$\|h z_h\|_{1,h}$ can be included in the norm above.  We
now prove a fundamental stability result for the discretization
\eqref{stabFEM}.
\begin{theorem}\label{Thm:infsup}
Assume that $(\gamma_V \gamma_W) \leq (C_i c_{\mathcal{T}})^{-2}$. Then
there exists a positive constant $c_s$ independent of $\gamma_V,\,
\gamma_W$  such that 
there holds
\[
c_s \tn (x_h,y_h) \tn \leq \sup_{(v_h,w_h) \in \mathcal{V}_h} \frac{A_h[(x_h,y_h),(v_h,w_h) ]}{\tn (v_h,w_h) \tn}.
\]
\end{theorem}
\begin{proof}
First we recall the positivity
\[
|x_h|_{s_V}^2 + \|y_h\|_{s_W}^2 = A_h[(x_h,y_h),(x_h,-y_h) ].
\]
Then observe that by integrating by parts in the bilinear form
$a_h(\cdot,\cdot)$ and using the zero mean value property of the
approximation space we have
\[
a_h(x_h,w_h) = \sum_{F \in \mathcal{F}} \int_F [n_F  \cdot \nabla x_h]
\{w_h\} ~\mbox{d}s.
\]
Define the function $\xi_h \in W_h$ such that for every face $F \in
\mathcal{F}_i \cup \mathcal{F}_{\Gamma_C}$
$$
\overline{ \{\xi_h\}}\vert_{F} := \gamma_V h_F  [n_F \cdot \nabla x_h]\vert_F.
$$
This is possible in the
nonconforming finite element space since the degrees of freedom may be
identified with the average value of the finite element function on an
element face. 
%Similarly define
%$\zeta_h \in V_h$ such that for every face $F \in
%\mathcal{F}_i \cup \mathcal{F}_{\Gamma_D'}$
%$$
%|F|^{-1} \int_F \{\zeta_h\} ~\mbox{d}s :=\frac12 h_F  [\nabla y_h \cdot n_F].
%$$
Using Lemma \ref{lem:invtrace} we have
\begin{equation}\label{eq:L2_stab_func}
%\|\zeta_h\|^2_{\kappa} \leq c_{\mathcal{T}} \sum_{F \in \partial \kappa}
%h^2_F \|h_F^{\frac12}  [\nabla y_h \cdot n_F]\|_F^2. %\mbox{ and }followed
 \|h^{-1} \xi_h\|^2_{\Omega} \leq c_{\mathcal{T}}^2   \sum_{F \in \mathcal{F}_{i}
   \cup \mathcal{F}_C}  \gamma_V^2 \|h_F^{\frac12}  [ n_F \cdot \nabla x_h]\|_F^2.
\end{equation}
Testing with $w_h = \xi_h$ and $v_h = 0$ we get
\begin{equation*}
\gamma_V \|h^{\frac12}[\nabla x_h \cdot n_F]\|_{\mathcal{F}_i \cup
  \mathcal{F}_{\Gamma_C}}^2 
%+
%\|h^{\frac12}[\nabla y_h \cdot n_F]\|_{\mathcal{F}_i \cup
%  \mathcal{F}_{\Gamma_D'}}^2\\ 
=
 A_h[(x_h,y_h),(0,\xi_h) ] + s_W(y_h,\xi_h). %-s(x_h,\zeta_h).
\end{equation*}
For the stabilization terms in the right hand side we have the upper
bounds, using the inverse inequality (trace inequality if
\eqref{eq:stab_dual_consist} is used) \eqref{trace}(ii) and \eqref{eq:L2_stab_func}
\begin{multline*}
s_W(y_h,\xi_h) \leq \|y_h\|_{s_W} \|\xi_h\|_{s_W} \leq C_i  \|y_h\|_{s_W}
\gamma_W^{\frac12} \|h^{-1} \xi_h\|_\Omega
\\ \leq C_i c_{\mathcal{T}}  \|y_h\|_{s_W} (\gamma_V \gamma_W)^{\frac12}
\|\gamma_V^{\frac12} h^{\frac12}  [n \cdot \nabla x_h ]\|_{\mathcal{F}_i \cup
  \mathcal{F}_{\Gamma_C}}.
\end{multline*}
%and
%\[
%s(y_h,\xi_h) \leq C_s  |x_h|_s \|\gamma^{\frac12} h^{\frac12}  [\nabla y_h \cdot n_F]\|_{\mathcal{F}_i \cup
% \mathcal{F}_{\Gamma_C}}.
%\]
The consequence of this is that for $\gamma_V \gamma_W < (C_i c_{\mathcal{T}})^{-2} $ there holds
\begin{multline}\label{first_stab}
\frac12 \left(|x_h|_{s_V}^2 + \|y_h\|_{s_W}^2 +
  \gamma_V^{\frac12}\|h^{\frac12}[n \cdot \nabla x_h ]\|_{\mathcal{F}_i \cup
  \mathcal{F}_{\Gamma_C}}^2 \right)
\\
\leq  A_h[(x_h,y_h),(x_h,-y_h + \xi_h) ].
\end{multline}
To include the control of the gradient of $x_h$ we use a well-known discrete Poincar\'e
inequality for piecewise constant functions \cite{EGH02}
\[
\|\nabla x_h\|_h^2 \leq C\sum_{F \in \mathcal{F}_{i}
   \cup \mathcal{F}_C}  h_F^{-1}\|[\nabla x_h]\|_F^2.
\]
%Consider the auxiliary
%problem:find $\varphi: \Omega \mapsto \mathbb{R}^d$
%\[
%\begin{array}{rcl}
%\Delta \varphi &=& \nabla x_h \mbox{ in } \Omega\\
%\varphi &=& 0 \mbox{ on } \tilde \Gamma_C\\
%n \cdot\nabla \varphi & = & 0 \mbox{ on } \partial \Omega\setminus
%\tilde \Gamma_C.
%\end{array}
%\]
%Since $\Omega$ is convex polyhedral and $\tilde \Gamma_C$ is a facet
%of $\Omega$, we known that the solution $\varphi$ satisfies the
%regularity bound (see \cite{Gris85})
%\begin{equation}\label{upper_bound_H2}
%\|\varphi\|_{H^2(\Omega)} \leq C_R \|\nabla x_h\|_\Omega.
%\end{equation}
%It follows using integration by parts that
%\begin{multline*}
%\[
%\|\nabla x_h\|_h^2 = (\nabla x_h, \Delta \varphi)_\Omega 
%= \sum_{F \in
%  \mathcal{F}} \int_{F} [\nabla x_h] \cdot (n_{F}  \nabla \varphi)  ~\mbox{d}s\\
%\leq \sum_{F \in
%  \mathcal{F}_i \cup \mathcal{F}_{\Gamma_C}} \int_{F}  (  |[n_{F} \cdot \nabla
%x_h]|+  [(I - n_{F} \otimes n_{F}) \nabla x_h] | )|\nabla
%\varphi|~\mbox{d}s.
%\]
%\end{multline*}
The right hand side is now upper bounded by decomposing the jump of
the gradient on its normal and tangential part and applying
the inverse inequality $$\|h^{\frac12} [(I - n_{F} \otimes n_{F}) \nabla x_h]\|_F
\leq C \|h^{-\frac12} [ x_h]\|_F$$ in the latter. Relating the right
hand side to the quantities in $\tn \cdot\tn _V$ already controlled in
\eqref{first_stab}, this leads to the upper bound
\[
\|\nabla x_h\|_h \leq C h^{-1} (\|h^{\frac12}  [n \cdot \nabla
x_h]\|_{\mathcal{F}_i \cup \mathcal{F}_{\Gamma_C}}+
\gamma_V^{-\frac12} |x_h|_{s_V}).
\]
and hence
\[
h \gamma_V^{\frac12} \|\nabla x_h\|_h\leq C (\gamma_V^{\frac12} \|h^{\frac12}  [n \cdot \nabla
x_h]\|_{\mathcal{F}_i \cup \mathcal{F}_{\Gamma_C}}+ |x_h|_{s_V}).
\]
We may conclude that there exists a positive constant
$c_0>0$ independent of $\gamma_V, \, \gamma_W$ such that
\[
c_0 \tn (x_h,y_h) \tn^2 \leq  A_h[(x_h,y_h),(x_h,-y_h + \xi_h) ].
\]
To end the proof we need to prove the stability of $\xi_h$ in
the triple norm. By the triangular inequality
\[
\tn (x_h , -y_h + \xi_h)\tn \leq \tn (x_h,y_h ) \tn + \tn (0, \xi_h) \tn.
\]
Using now an inverse inequality followed by the
argument of \eqref{eq:L2_stab_func} 
we arrive at
\begin{equation*}
\tn (0 , \xi_h)\tn  = \|\xi_h\|_{s_W} \leq \gamma_W^{\frac12} C_i
\|h^{-1} \xi_h\|_{\Omega} %\\
\leq C_i c_{\mathcal{T}}(\gamma_W\gamma_V)^{\frac12} \tn x_h \tn_V
\leq \tn (x_h,y_h) \tn.
\end{equation*}
This concludes the proof with $c_s = c_0/2$.
\end{proof}
\begin{remark}\label{rem:weak_stab}
If the stabilization operator defined by equation \eqref{eq:stab_dual_consist} is used,
stability of $\|h \nabla y_h\|$ may be included using a similar
argument. This control of the dual variable is nevertheless weaker
than that provided using \eqref{stab_LS}.
\end{remark}
\begin{corollary}
The formulation \eqref{stabFEM} admits a unique solution $(u_h,z_h)$.
\end{corollary}
\begin{proof}
The system matrix corresponding to \eqref{stabFEM} is a square matrix
and we only need to show that there are no zero eigenvalues.
Assume that $l(w_h) = 0$. It then follows by Theorem \ref{Thm:infsup}
that for any solution $(u_h,z_h)$ there holds
\[
c_s \tn (u_h,z_h) \tn \leq \sup_{(v_h,w_h) \in \mathcal{V}_h}
\frac{A_h[(u_h,z_h),(v_h,w_h) ]}{\tn (v_h,w_h) \tn} = 0,
\]
implying that $u_h=0$, $z_h=0$ which shows that the solution is unique.
\end{proof}
\section{Error estimates}\label{sec:error}
Even though Theorem \ref{Thm:infsup} provides us with a stability
estimate for the formulation, the norm is not sufficiently strong to
allow for a proof of convergence. Indeed the only notion of stability
at our disposal that can allow us to prove error estimates are
\eqref{stab_1} and \eqref{stab_2}. We will follow the approach
introduced in \cite{Bu14} and first prove that $\tn (u - u_h, z_h) \tn
\leq C h |u|_{H^2(\Omega)}$. This tells us that the stabilization
terms must vanish at an optimal rate for smooth $u$ and that $\|\nabla
u_h\|_h + \| \nabla z_h \|_h$ is uniformly bounded as $h \rightarrow
0$. Using this a priori bound we may conclude that the
$H^1$-conforming part of $u_h$ is uniformly bounded in $H^1$. This
allows us to write the error $u - u_h$ as $u - \tilde u_h + \tilde u_h
- u_h = \tilde e + e_h$, where $\tilde u_h$ denotes the $H^1$-conforming part of $u_h$. We may then control the part $\tilde e$
using the continuous dependence estimates \eqref{stab_1} and
\eqref{stab_2}, while $e_h$ is shown to be bounded by the
stabilization.
\begin{proposition}\label{prop:stab_conv}
Let $u \in H^2(\Omega)$ be the solution of \eqref{eq:Pb} and
$(u_h,z_h) \in \mathcal{V}_h$ the solution of \eqref{stabFEM}. Then
\begin{equation}\label{tn_bound}
\tn(u-u_h,z_h)\tn \leq C (\gamma_V^{\frac12}+ c_s^{-1}  (\gamma_W^{-\frac12} +\gamma_V^{\frac12} )) h\|u\|_{H^2(\Omega)}
\end{equation}
and 
\begin{equation}\label{H1apriori}
\|\nabla u_h\|_h \leq  C(1+ c_s^{-1}  (\gamma_W^{-\frac12}  \gamma_V^{-\frac12}+1 )) \|u\|_{H^2(\Omega)}.
\end{equation}
\end{proposition}
\begin{proof}
Using a triangle inequality and the approximation \eqref{tnorm_approx} it is
sufficient to consider the discrete error $\mu_h = u_h - r_h u$. By Theorem \ref{Thm:infsup} we have the
stability
\begin{equation}\label{bas_stab}
c_s \tn(\mu_h, z_h)\tn \leq \sup_{(v_h,w_h) \in \mathcal{V}_h} \frac{A_h[(\mu_h,z_h),(v_h,w_h) ]}{\tn (v_h,w_h) \tn}.
\end{equation}
Using the formulation and Lemma \ref{lem:rhH1proj} we observe that
\begin{multline*}
A_h[(\mu_h,z_h),(v_h,w_h) ] =
a_h(\mu_h,w_h)- s_W(z_h,w_h) + a_h(v_h,z_h) + s_V(\mu_h,v_h)\\
= a_h(u_h - u,w_h)- s_W(z_h,w_h)-s_V(r_h u,v_h).
\end{multline*}
Applying Lemma \ref{lem:weakcons} to the right hand side with $\nu_h :=
r_h u$ we have
\begin{equation}
|A_h[(\mu_h,z_h),(v_h,w_h) ]| \leq  \sum_{F \in \mathcal{F}_i \cup \mathcal{F}_{\Gamma_C'}} \int_{F}
|(\nabla u - \{\nabla r_h u\})\cdot n_F || [w_h]| ~\mbox{d}s +
|s_V(r_h u, v_h)|.
\end{equation}
We proceed using the Cauchy-Schwarz inequality followed by an element
wise trace inequalities and the approximation \eqref{eq:approx} to
obtain
\begin{multline*}
\sum_{F \in \mathcal{F}_i \cup \mathcal{F}_{\Gamma_C'}} \int_{F}
|(\nabla u - \{\nabla r_h u\})\cdot n_F || [w_h]| ~\mbox{d}s  +
|s_V(r_h u, v_h)|\\ \leq
C \left(\sum_{F \in \mathcal{F}_i \cup \mathcal{F}_{\Gamma_C'}}
 \gamma_W^{-1} \|h_F^{\frac12} (\nabla u - \{\nabla r_h u\})\cdot n_F
  \|^2_F\right)^{\frac12} \|w_h\|_{s_W} + |u - r_h u|_{s_V} |v_h|_{s_V} \\
\leq C(\gamma_W^{-\frac12} +\gamma_V^{\frac12} ) h \|u\|_{H^2(\Omega)} \tn (v_h,w_h) \tn.
\end{multline*}
Applying the above inequalities in \eqref{bas_stab} completes the
proof of \eqref{tn_bound}. The inequality \eqref{H1apriori} then is an
immediate consequence of \eqref{tn_bound} and the $H^1$-stability of
$r_h$.
\begin{multline*}
\|\nabla u_h\|_h \leq \|\nabla \mu_h\|_h+ \|\nabla r_h u\|_h\\
 \leq  C (\gamma_V^{-\frac12} 
h^{-1} \tn(\mu_h, z_h)\tn + \|u\|_{H^1(\Omega)}) \leq C(1+ c_s^{-1}  (\gamma_W^{-\frac12} +\gamma_V^{\frac12} ) \gamma_V^{-\frac12} ) \|u\|_{H^2(\Omega)}.
\end{multline*}
\end{proof}
\begin{theorem}\label{Thm:cont_dep_error_est}
Let $u \in H^2(\Omega)$ be the solution of \eqref{eq:Pb} and
$(u_h,z_h) \in \mathcal{V}_h$ the solution of \eqref{stabFEM}. Then,
with $j(\cdot)$ and $\Xi(\cdot)$ defined in \eqref{stab_1} or
\eqref{stab_2}, there exists $h_0<0$ and constant $C>0$ independent of
$h$ such that for all $h<h_0$
\[
|j(u - u_h)| \leq \Xi( \eta(h,l,u_h,z_h)) + C \gamma_V^{-\frac12} h |u_h|_{s_V}
\]
where 
\[
\begin{aligned}
\eta(h,l,u_h,z_h) &= C (h \|f\|_{\Omega} +  \gamma_V^{-\frac12} |u_h|_{s_V} +
\gamma_W^{\frac12} \|z_h\|_{s_W})\\& + C\left(\sum_{F \in
  \mathcal{F}_{\Gamma_C}} h \inf_{\alpha_F \in \mathbb{R}} \|\psi -
\alpha_F\|^2_F \right)^{\frac12}.
\end{aligned}
\]
In addition the following a priori bound holds
\[
\eta(h,l,u_h,z_h) +   |u_h|_{s_V} \leq C h
(\|f\|_{\Omega}+\|\psi\|_{H^{\frac12}(\Gamma_C)} + \|u\|_{H^2(\Omega)}),
\]
where the constant includes that of \eqref{tn_bound}.
\end{theorem}
\begin{proof}
By the definition $j(\cdot)$ is an $L^2$-norm and therefore well defined for functions in
$V+V_h$. We then consider the decomposition of $u - u_h$ into one
$V$-conforming part and its residual. To this end introduce a
function $\tilde u_h \in V \cap V_h$. To get an $H^1$-conforming
approximation we define the values of $\tilde u_h$ in the vertices
$x_i$ of the tesselation $\mathcal{T}_h$ by $\tilde u_h\vert_{\bar
  \Gamma_D} = 0$ and,
\begin{equation}\label{quasi_interp}
\tilde u_h(x_i) = 
\textfrak{C}_{x_i}^{-1} \sum_{\kappa: x \in \kappa}
u_h(x_i)\vert_{\kappa},\quad  x_i \not \in \bar \Gamma_D,
\end{equation}
where $\textfrak{C}_{x_i}:=\mbox{card}(\{\kappa \in \mathcal{T}_h: x_i \in
\kappa\})$. With this definition it holds that $\tilde u_h  \in V \cap
V_h$. For the discrete error $e_h :=  u_h - \tilde u_h $ it is well known that the following estimate holds (see \cite{ABC03,KP03})
\begin{equation}\label{discrete_interpol}
\|e_h\| + h \|\nabla e_h\|_h \leq C h \gamma_V^{-\frac12} |u_h|_{s_V}.
\end{equation} 
We may then construct the $H^1$-conforming part of the error as $\tilde e := u - \tilde u_h \in V$, making it a
valid function to use in the continuous dependence
\eqref{cont_dep_assump}. 
For any $w \in W$ there holds
\[
a(\tilde e, w) = l(w) - a(\tilde u_h,w) =: \left< r,w\right>_{W',W}
\]
where we have identified $r \in
W'$.
To apply \eqref{cont_dep_assump} we need to upper bound $\|r\|_{W'}$,
this follows by
\begin{multline*}
\sup_{\substack{w \in W\\
\|w\|_{H^1(\Omega)}=1}} \left< r,w\right>_{W',W} = \sup_{\substack{w \in W\\
\|w\|_{H^1(\Omega)}=1}} (l(w- r_h w)  
- a_h(e_h,w) \\
- s_W(z_h,r_h w) -\underbrace{a_h(u_h,w - r_h w)}_{=0})
\end{multline*}
where the last term vanishes similarly as in Lemma \ref{lem:rhH1proj}.
For the second to last term there holds by the Cauchy-Schwarz
inequality (and approximation and a trace inequality if
\eqref{eq:stab_dual_consist} is used) and the $H^1$-stability of $r_h$,
\[
|s_W(z_h,r_h w)| \leq C \|z_h\|_{s_W} \gamma_W^{\frac12}\|\nabla r_h w\|_h \leq C  \gamma_W^{\frac12}\|z_h\|_{s_W}.
\]
Using Cauchy-Schwarz inequality and the discrete interpolation result
\eqref{discrete_interpol}
we obtain for the second term in the right hand side
\[
 a_h(e_h,w) \leq \|e_h\|_{1,h} \|w\|_{H^1(\Omega)} \leq C
 \gamma_V^{-\frac12} |u_h|_{s_V}.
\]
By the definition of $l(\cdot)$ we see that the first term on the
right hand side may be
bounded by
\begin{multline*}
l(w- r_h w) = (f, w- r_h w)_\Omega + \sum_{F \in \mathcal{F}_{\Gamma_C}} (\psi - \alpha_F, w- r_h
w)_{F} \\
 \leq C h \|f\|_{\Omega} + C \left(\sum_{F \in
  \mathcal{F}_{\Gamma_C}} h \inf_{\alpha_F \in \mathbb{R}} \|\psi -
\alpha_F\|^2_F \right)^{\frac12}.
\end{multline*}
It follows from \eqref{tn_bound} and standard approximation that for $h$ small enough, $\tilde e$ satisfies the assumptions of the continuous
dependence \eqref{cont_dep_assump}. However note that in order to apply \eqref{stab_1} or
\eqref{stab_2} to
$\tilde e$ we must show that there exists $E>0$ such that the bound $\|\tilde e\|_{H^1(\Omega)} \leq E
< \infty$ holds uniformly in $h$, since otherwise the constants in the
estimates may blow up. This a priori bound is a consequence of a triangle inequality,
\eqref{H1apriori} and
the estimate \eqref{discrete_interpol} as follows 
\begin{multline*}
\|\tilde e\|_{H^1(\Omega)} \leq \|u\|_{H^1(\Omega)} + \|u_h\|_{1,h}
+\|e_h\|_{1,h} \\
\leq   \|u\|_{H^1(\Omega)} + \|u_h\|_{1,h}+
 C  \gamma_V^{-\frac12} |u_h|_{s_V} \leq C(1+h) \|u\|_{H^2(\Omega)}.
\end{multline*}
Therefore, under our regularity assumption on the exact solution,
 the $H^1$-norm of the conforming part of the error is uniformly
bounded for all $h$.
For the case of \eqref{stab_1} or \eqref{stab_2} we note that for all
$\omega \subset \Omega$ there holds
\[
\|u - u_h\|_{\omega} \leq \|\tilde e\|_{\omega} + \|e_h\|_{\omega}
\leq \Xi( \eta(h,l,u_h,z_h)) + C \gamma_V^{-\frac12}h |u_h|_{s_V}
\]
where $\Xi(\cdot)$ is defined by \eqref{stab_1} or \eqref{stab_2}
depending on the choice of $\omega$. The upper bounds on
$\eta(h,l,u_h,z_h)$ and $|u_h|_{s_V}$ are immediate consequences of
Proposition \ref{prop:stab_conv} and the approximation properties of
piecewise constant functions.
\end{proof}
\subsection{The case of perturbed data}\label{sec:perturbed_data}
Very often in applications the problem under study is a Poisson equation
that it is reasonable to believe is well posed and the solution of
which satisfies a standard stability estimate and regularity in $H^2(\Omega)$. The problem is that the boundary conditions on
$\Gamma_D'$ are unknown. Instead we have at our disposal measurements
of the fluxes $\psi + \delta \psi$ on the boundary part
$\Gamma_C$. These measurements are usually polluted by measurement
errors, $\delta \psi$. It is then of interest to study how fine it is
reasonable to make the mesh, knowing that the perturbed data might not
be in the range of the operator. The perturbed problem may be written, find
$u_{\delta} \in V$ such that 
\begin{equation}\label{pert_Cauchy_prob}
a(u_{\delta},v) =  l_{\delta}(w) :=  l(w)+ \delta l(w)
\end{equation}
where 
\[
\delta l(w) :=  \int_{\Omega} \delta f w ~\mbox{d}x +
 \int_{\Gamma_C} \delta \psi w ~\mbox{d}s.
\]
We will assume that $\delta f \in L^2(\Omega)$ and $\delta \psi \in
L^2(\Gamma_C)$ and introduce the h-weighted dual norm,
\[
\|(\delta f,\delta \psi)\|_{h,W'} := h \|\delta f\|_{\Omega}+\|\delta f\|_{W'} +
h^{\frac12} \|\delta \psi\|_{\Gamma_C} + \|\delta \psi\|_{H^{-\frac12}(\Gamma_C)}.
\]
This norm will be used to measure the perturbation induced by errors
in measurements. The reason for the combination of strong and weak
norms is the following boundedness results.
\begin{lemma}\label{lem:pert_bound}
Let $s_W(\cdot,\cdot)$ be defined by \eqref{stab_LS}
\begin{equation}\label{disc_pert}
\sup_{\substack{w_h \in W_h\\
\| w_h\|_{s_W} = 1}} |l( w_h) - l_{\delta}(w_h)| \leq C
\gamma_{W}^{-\frac12} \|(\delta f,\delta \psi)\|_{h,W'}.
\end{equation}
\begin{equation}\label{cont_pert}
\sup_{\substack{w \in W\\
\|w\|_W = 1}} |l(r_h w) - l_{\delta}(r_h w)| \leq C \|(\delta f,\delta \psi)\|_{h,W'}.
\end{equation}
\end{lemma}
\begin{proof}
By definition $\delta l(w_h) = l(w_h) - l_{\delta}(w_h)$ and by the
linearity of the operator
\[
|\delta l(w_h)| \leq |\delta l(\tilde w_h)|+|\delta l(w_h - \tilde w_h)|,
\]
where $\tilde w_h \in W_h$ is the $H^1$-conforming part of $w_h$
defined similarly as in \eqref{quasi_interp}, but with $\tilde
w_h\vert_{\Gamma_C'} = 0$.
We may then use an estimate similar to \eqref{discrete_interpol}, but
with $\|\cdot\|_{s_W}$, to obtain the bounds
\begin{multline*}
|\delta l(\tilde w_h)| = |\left<\delta f,\tilde w_h \right>_{W',W} +
\left<\delta \psi,\tilde  w_h\right>_{H^{-\frac12},H^{\frac12}}| \\
\leq
C (\|\delta f\|_{W'}
+ \|\delta \psi\|_{{H^{-\frac12}}}) \|\tilde w_h\|_{H^1(\Omega)} \\
\leq
C (\|\delta f\|_{W'}
+ \|\delta \psi\|_{{H^{-\frac12}}}) 
( \|\tilde w_h - w_h\|_{1,h} + \|
w_h\|_{1,h})\\
\leq C (\|\delta f\|_{W'}
+ \|\delta \psi\|_{{H^{-\frac12}}}) \gamma_{W}^{-\frac12} \| w_h \|_{s_W}
\end{multline*}
and,
\begin{multline*}
|\delta l(w_h - \tilde w_h)| \leq \|f\|_{\Omega} \|w_h - \tilde
w_h\|_{\Omega} + \|\psi\|_{\Gamma_C}  \|w_h - \tilde
w_h\|_{\Gamma_C} \\
\leq C( h \|f\|_{\Omega} + h^{\frac12} \|\psi\|_{\Gamma_C}) \gamma_{W}^{-\frac12} \|w_h\|_{s_W}.
\end{multline*}
Similarly the bound on $|\delta l(r_h w)|$ is obtained by
\[
|\delta l(r_h w)| =  |\delta l(r_h w - w)+ \delta l(w)| \leq C \|(\delta f,\delta \psi)\|_{h,W'}
\]
where we used the approximation \eqref{eq:approx} with $t=1$ and the duality pairing $\delta
l(w) = \left<\delta f,w \right>_{W',W} + \left<\delta \psi, w\right>_{H^{-\frac12},H^{\frac12}}$.
\end{proof}

\begin{remark}
The Lemma \ref{lem:pert_bound} only holds when the stabilization of
\eqref{stab_LS} is used in \eqref{stabFEM}. If instead \eqref{eq:stab_dual_consist} is used, one
may only obtain control of $\|h \nabla w_h\|_h$ in the triple norm (see Remark
\ref{rem:weak_stab}), leading to an
additional factor $h^{-1}$ in the right hand side of \eqref{disc_pert}
above.
\end{remark}

Accounting for the perturbed data introduces a minor modification of the weak consistency that holds
for the formulation \eqref{stabFEM}, when the right hand side is
substituted for the perturbed functional $l_\delta(w_h)$.
\begin{lemma}\label{lem:weakcons_pert}(Weak consistency with perturbed
  data)
Let $u$ be the solution of \eqref{eq:Pb}, with $f \in L^2(\Omega)$ and
$\psi \in H^{\frac12}(\Gamma_{C})$ and let $(u_h,z_h) \in \mathcal{V}_h$ be
the solution of \eqref{stabFEM} with the right hand side given by
$l_\delta(w_h)$. Then,  for all $w_h \in W_h$, there holds,
\begin{multline}\label{galorth_pert}
|a_h(u_h - u,w_h) - s(z_h,w_h)|
\leq \sum_{F \in \mathcal{F}_i \cup  \mathcal{F}_{\Gamma_C'}} \inf_{\nu_h \in V_h}\int_{F}
|(\nabla u - \{\nabla \nu_h\})\cdot n_F || [w_h]| ~\mbox{d}s \\
+|\delta l(w_h)|.
\end{multline}
\end{lemma}
\begin{proof}
Following the proof of Lemma \ref{lem:weakcons} we now find that
\[
a_h(u_h - u,w_h) - s(z_h,w_h)
= -\sum_{\kappa \in \mathcal{T}_h}
\sum_{\substack{F \in \partial \kappa \\ F \not \in \Gamma_C}} \int_F \nabla u \cdot n_{\kappa} w_h ~
\mbox{d}s + \delta l(w_h).
\]
We conclude as in Lemma \ref{lem:weakcons}.
\end{proof}
It is then straightforward to derive modified versions of Proposition
\ref{prop:stab_conv} and Theorem \ref{Thm:cont_dep_error_est}. We give
the results for the perturbed case below, detailing only the parts of
the proofs that are modified by the perturbed right hand side in
\eqref{stabFEM}. Observe that if the problem \eqref{pert_Cauchy_prob}
admits a solution $u_\delta \in H^2(\Omega)$, then the Proposition
\ref{prop:stab_conv} still holds if $u$ is exchanged with
$u_\delta$. If on the other hand \eqref{pert_Cauchy_prob} does not
have a solution, or $\|u_\delta\|_{H^2(\Omega)}$ is very large, the perturbation can be included in the following way.
\begin{proposition}\label{prop:stab_conv_pert}
Let $u \in H^2(\Omega)$ be the solution of \eqref{eq:Pb} and
$(u_h,z_h) \in \mathcal{V}_h$ the solution of \eqref{stabFEM} using
\eqref{stab_LS} and with
the perturbed right hand side $l_\delta(w_h)$. Then
\begin{equation}\label{tn_bound_pert}
\tn(u-u_h,z_h)\tn \leq  C ((\gamma_V^{\frac12}+ c_s^{-1}  (\gamma_W^{-\frac12}
+\gamma_V^{\frac12} )) h \|u\|_{H^2(\Omega)} +  c_s^{-1}  \gamma_{W}^{-\frac12} \|(\delta f, \delta \psi)\|_{h,W'})
\end{equation}
and 
\begin{equation}\label{H1apriori_pert}
\|\nabla u_h\|_h \leq C ((1+ c_s^{-1}  (\gamma_W^{-\frac12}\gamma_V^{-\frac12} 
+1)) \|u\|_{H^2(\Omega)} +  c_s^{-1} \gamma_{W}^{-\frac12}  h^{-1} \|(\delta f, \delta \psi)\|_{h,W'}).
\end{equation}
\end{proposition}
\begin{proof}
The proof follows the arguments of the proof of Proposition
\ref{prop:stab_conv}, but this time we use the modified
weak consistency of Lemma \ref{lem:weakcons_pert}
\begin{multline}\label{pert_up_b1}
|A_h[(\mu_h,z_h),(v_h,w_h) ]| \leq  \sum_{F \in \mathcal{F}_i \cup \mathcal{F}_{\Gamma_C'}} \int_{F}
|(\nabla u - \{\nabla r_h u\})\cdot n_F || [w_h]| ~\mbox{d}s +|\delta
l(w_h)|\\
+ |s_V(r_h u, v_h)|.
\end{multline}
The second term of the right hand side is then bounded using
inequality \eqref{disc_pert}. The bound \eqref{H1apriori_pert} follows
as before using the definition of the norm $\tn w_h \tn_V$ and the
estimate \eqref{tn_bound_pert}.
\end{proof}
 We observe that the uniform $H^1$-bound on $u_h$ no longer
 holds. Indeed since it can not be assumed that the solution
 $u_\delta$ of the perturbed problem \eqref{pert_Cauchy_prob} exists
 the method can fail to converge in the limit $h \rightarrow
 \infty$. Assuming that the contribution from the discretization error
 dominates the upper bound \eqref{tn_bound_pert} an error estimate in
 the spirit of Theorem \ref{Thm:cont_dep_error_est} can nevertheless be derived.
\begin{theorem}\label{Thm:cont_dep_error_pert}
Let $u \in H^2(\Omega)$ be the solution of \eqref{eq:Pb} and
$(u_h,z_h) \in\mathcal{V}_h$ the solution of \eqref{stabFEM} using
\eqref{stab_LS} and with
the perturbed right hand side $l_\delta(w_h)$. Assuming that there
exists $h_0>0$ such that
\begin{equation}\label{small_pert}
\max(1, \gamma_W^{-\frac12})\|(\delta f, \delta \psi)\|_{h,W'} \leq  h_0 \|u\|_{H^2(\Omega)} 
\end{equation}
and
\[
\begin{aligned}
\eta_{\delta}(h,l,u_h,z_h) &= C (h \|f\|_{L^2(\Omega)} +  |u_h|_{s_V} +  \|z_h\|_{s_W}\\ &+ \left(\sum_{F \in
  \mathcal{F}_{\Gamma_C}} h \inf_{\alpha_F \in \mathbb{R}} \|\psi -
\alpha_F\|^2_F \right)^{\frac12}+ \|(\delta f, \delta \psi)\|_{h,W'} )<
1.
\end{aligned}
\]
for $h<h_0$. Then, with $j(\cdot)$ and $\Xi(\cdot)$ defined in \eqref{stab_1} or
\eqref{stab_2} we have
\begin{equation}\label{up_bound_pert}
|j(u - u_h)| \leq \Xi( \eta_\delta(h,l,u_h,z_h)) + C h |u_h|_{s_V}.
\end{equation}
In addition the following a priori bound holds
\[
\eta_{\delta}(h,l,u_h,z_h) +  |u_h|_{s_V}  \leq C h
(\|f\|+\|\psi\|_{H^{\frac12}(\Gamma_C)}) + C h_0 \|u\|_{H^2(\Omega)},
\]
where the constant includes that of \eqref{tn_bound_pert}.
\end{theorem}
\begin{proof}
Under the assumption \eqref{small_pert} the proof is analoguous to
that of Theorem \ref{Thm:cont_dep_error_est}, since by \eqref{small_pert}
equations \eqref{tn_bound_pert} and \eqref{H1apriori_pert} take the
same form as \eqref{tn_bound} and \eqref{H1apriori}. This means that
$\|\tilde e\|_{H^1(\Omega)}$ is uniformly bounded in $h$ under the
condition \eqref{small_pert} and therefore the constants in
\eqref{stab_1} and \eqref{stab_2} remain bounded. The only
difference in the proof appears in the estimation of the residual
term $r \in W'$, here
\begin{multline*}
\sup_{\substack{w \in W\\
\|w\|_{H^1(\Omega)}}} \left< r,w\right>_{W',W} = \sup_{\substack{w \in W\\
\|w\|_{H^1(\Omega)}}} (l(w- r_h w)  - \underbrace{\delta l(r_h w)}_{perturbation}
+ a_h(e_h,w) \\
- s_W(z_h,r_h w) -\underbrace{a_h(u_h,w - r_h w)}_{=0}).
\end{multline*}
The new contribution is the second term of the right hand side
due to the perturbed data. This term is upper bounded using
\eqref{cont_pert} and the result follows.
\end{proof}
We see that the estimate only is valid when $\|(\delta f, \delta
\psi)\|_{h,W'}$ is small compared to   $h \|u\|_{H^2(\Omega)}$. This
is not a very useful condition in practice since $\|u\|_{H^2(\Omega)}$
is unknown. However, assuming that $\|(\delta f, \delta
\psi)\|_{h,W'}$ is known, the quantities
that form the upper bound \eqref{up_bound_pert} are all computable, without any need to assume
additional regularity of the solution. Indeed $\eta_\delta(h,l,u_h,z_h) $
can be computed and the bound \eqref{small_pert} is necessary only to
ensure that the $H^1$-norm of $\tilde e$ stays bounded. This quantity
can also be controlled a posteriori using
\eqref{discrete_interpol}. It follows from Theorem \ref{Thm:cont_dep_error_pert} that mesh refinement will improve the solution as long as $\|\nabla u_h\|_h$ stays bounded, 
$|u_h|_{s_V} + \|z_h\|_{s_W}$ decreases and
\[
h \|f\|_{L^2(\Omega)} +  |u_h|_{s_V} +  \|z_h\|_{s_W} + \left(\sum_{F \in
  \mathcal{F}_{\Gamma_C}} h \inf_{\alpha_F \in \mathbb{R}} \|\psi -
\alpha_F\|^2_F \right)^{\frac12} >\|(\delta f, \delta \psi)\|_{h,W'}.
\]
\section{Numerical example}\label{sec:numerical}
 In this section we have used a version
of \eqref{stabFEM} where a consistent penalty method is used for the
imposition of the boundary conditions. This leads to a weak implementation of boundary conditions
reminiscent of Nitsche's method, for which the above analysis holds
after minor modifications. This procedure can be very useful, since
many finite element packages can not impose different Dirichlet
conditions on the trial and test spaces. The formulation with consistent penalty imposition of the boundary
conditions reads find $(u_h,z_h) \in X_h^\emptyset \times X_h^\emptyset$ such that,
\begin{equation}\label{stabFEM_weak}
\begin{array}{rcl}
a_h(u_h,w_h) - s_W(z_h,w_h) &=&l(w_h) - \displaystyle \sum_{F \subset \Gamma_D} \int_{F }
\nabla w_h \cdot n_\kappa g ~\mbox{d}s\\[3mm]
a_h(v_h,z_h) + s_V(u_h,v_h)  &=& \displaystyle  \sum_{F \subset \Gamma_D} \int_{F }\gamma_V h_F^{-1} g v_h ~\mbox{d}s
\end{array} 
\end{equation}
for all $(v_h,w_h) \in X_h^\emptyset \times X_h^\emptyset$. Here $g$
denotes some Dirichlet data on $\Gamma_D$ and
$X_h^\emptyset$ denotes the nonconforming finite element space with no boundary conditions 
imposed and the bilinear forms are modified as follows
\begin{multline}\label{weak_ah}
a_h(v_h,w_h) := \sum_{\kappa \in \mathcal{T}_h} \int_\kappa \nabla v_h
\cdot \nabla w_h ~ \mbox{d} x\\
- \sum_{F \in \partial \Omega} \left( \int_{F \cap \Gamma_C'}
\nabla v_h \cdot n_\kappa w_h ~\mbox{d}s + \int_{F \cap \Gamma_D}
\nabla w_h \cdot n_\kappa v_h ~\mbox{d}s\right),
\end{multline}
\begin{equation}\label{eq:num_stab1}
s_W(z_h,w_h)  := \sum_{\kappa \in \mathcal{T}_h} \int_\kappa \gamma_W
\nabla z_h\cdot \nabla w_h~\mbox{d}x + \sum_{F \in \mathcal{F}_{\Gamma_C'}} \int_{F} \gamma_{bc} h_F^{-1} z_h w_h ~\mbox{d}s,
\end{equation}
or 
\begin{equation}\label{eq:weak_dual_stab_bc}
s_W(z_h,w_h)  := \sum_{F \in \mathcal{F}_i} \int_{F} \gamma_W
h_F^{-1} [z_h][w_h] ~\mbox{d}s+  \sum_{F \in \mathcal{F}_{\Gamma_C'}} \int_{F} \gamma_{bc} h_F^{-1} z_h w_h ~\mbox{d}s.
\end{equation}
The stabilization term $s_V(\cdot,\cdot)$ of equation \eqref{stab_primal} is used
without modification.

As a numerical illustration of the theory we consider the original Cauchy
problem discussed by Hadamard. In \eqref{eq:Pb} let $\Omega:= (0,\pi)
\times (0,1)$, $\Gamma_C:=\{ x \in (0,\pi); y=0\}$, $\Gamma_D:=
\Gamma_C \cup \{x \in \{0,\pi \}; y \in (0,1)  \}$ and 
\begin{equation}\label{Cauchy_data}
\psi := A_n \sin(n x).
\end{equation}
It is then straightforward to verify that 
\begin{equation}\label{exact_Cauchy}
u_n = A_n n^{-1} \sin(nx) \sinh(n y)
\end{equation}
solves \eqref{eq:Pb}. One may easily show show that the
choice $A_n = n^{-p}$, $p > 0$ leads to $\psi \rightarrow 0$ uniformly
as $n \rightarrow \infty$,
whereas, for any $y>0$, $u_n(x,y)$ blows up. Stability can only be obtained
conditionally, using that $\|u_n\|_{H^1(\Omega)} < E$ for some $E>0$, leading to the
relations \eqref{stab_1} and \eqref{stab_2} (see
\cite{ARRV09} for detailed proofs and further discussion of
\eqref{cont_dep_assump}, \eqref{stab_1}. \eqref{stab_2}.)

We choose $A_n:=1$ in \eqref{Cauchy_data} and study the error in the
relative $L^2$-norms,
\begin{equation}\label{relative_L2}
\frac{\|u - u_h\|_{\Omega_\zeta}}{\|u\|_{\Omega_\zeta}}, \mbox{ where } \Omega_\zeta := (0,\pi)\times (0,\zeta),\quad \zeta \in \{1/8,\, 1/4,\, 1/2,\, 1\}.
\end{equation}
Recall that for $\zeta<1$ the stability \eqref{cont_dep_assump},
holds with \eqref{stab_1} and for $\zeta = 1$ \eqref{cont_dep_assump} with
\eqref{stab_2} holds. All computations below were performed using the
package FreeFEM++ \cite{He12}.
\subsection{Tuning of penalty parameters}
For $s_V$ we used a single penalty parameter $\gamma_V$, whereas
numerical experience showed that it is advantageous to use different
parameter in the interior and on the boundary for $s_W$.
We therefore
have three parameters to choose, $\gamma_V$, $\gamma_W$ and
$\gamma_{bc}$. We chose $n=3$ in \eqref{Cauchy_data} and studied the global ($\zeta=1$) relative 
$L^2$-error on an
unstructured mesh with $h \approx 0.1$ under the variation of the
stabilization parameters.

\subsubsection{Using the stabilization \eqref{eq:weak_dual_stab_bc} }
Numerical experimentation showed that the parameter $\gamma_{bc}$ had
to be set sufficiently big and we fixed it to $\gamma_{bc} =
100$. They also showed that $\gamma_V= \gamma_W$ was a reasonable
choice and we therefore varied the parameter $\gamma_V = \gamma_W$ in the interval $(0.0005, 0.5)$. The result
is shown in the left plot of Figure \ref{fig:parameter_study}. We see
that $\gamma_V=\gamma_W=0.01$ is a good choice for the parameter in this case. 
\subsubsection{Using the stabilization \eqref{eq:num_stab1} }
For the
method using the stabilization \eqref{eq:num_stab1},  numerical
experimentation showed that
$\gamma_V=\gamma_W$ was not a good choice and we therefore fixed
$\gamma_V = 0.01$ and varied $\gamma_W$ in the interval $(10^{-6}, 1.0)$. The result is shown in the right plot of Figure
\ref{fig:parameter_study}. 
\subsubsection{Further remarks and parameter choices}
The conclusions were that both methods are
relatively robust with respect to the variations of the penalty parameters, the error
remained under $10\%$ for a wide range of stabilization parameters on
this coarse mesh. Numerical experiments not reported here however showed that
the parameter $\gamma_W$ giving the minimum error in the right plot of
Figure
\ref{fig:parameter_study}, performed worse on finer meshes, in
particular when $n$ was increased. We therefore used a smaller value
of $\gamma_W$ in this case.

In the computations below we used $\gamma_{bc} = 100$ and either
\eqref{eq:weak_dual_stab_bc} with $\gamma_V = \gamma_W = 0.01$ or
\eqref{eq:num_stab1} with $\gamma_V = 0.01$ and $\gamma_W =
10^{-5}$. 
\begin{figure}
\includegraphics[height=6cm]{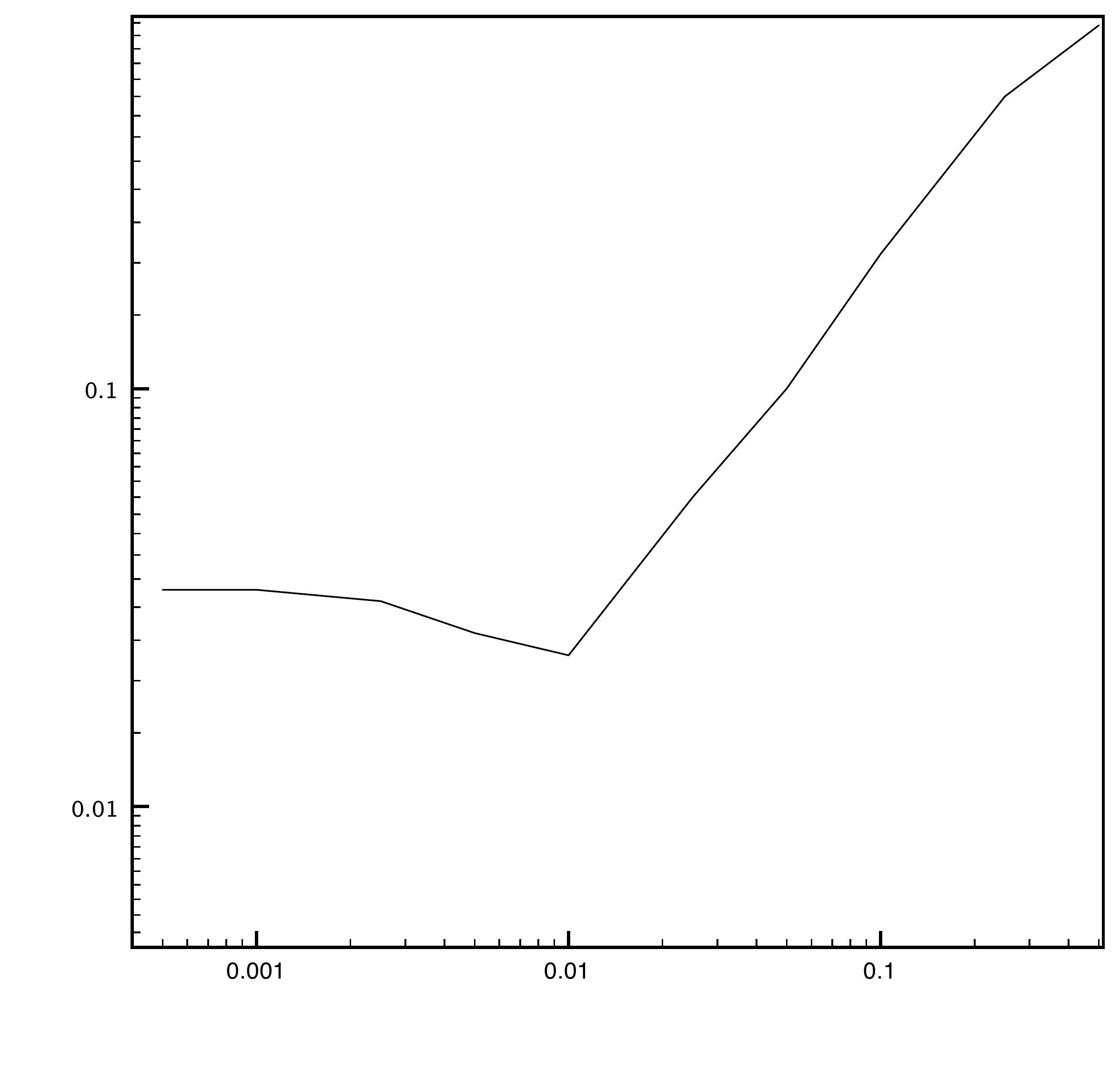}\includegraphics[height=6cm]{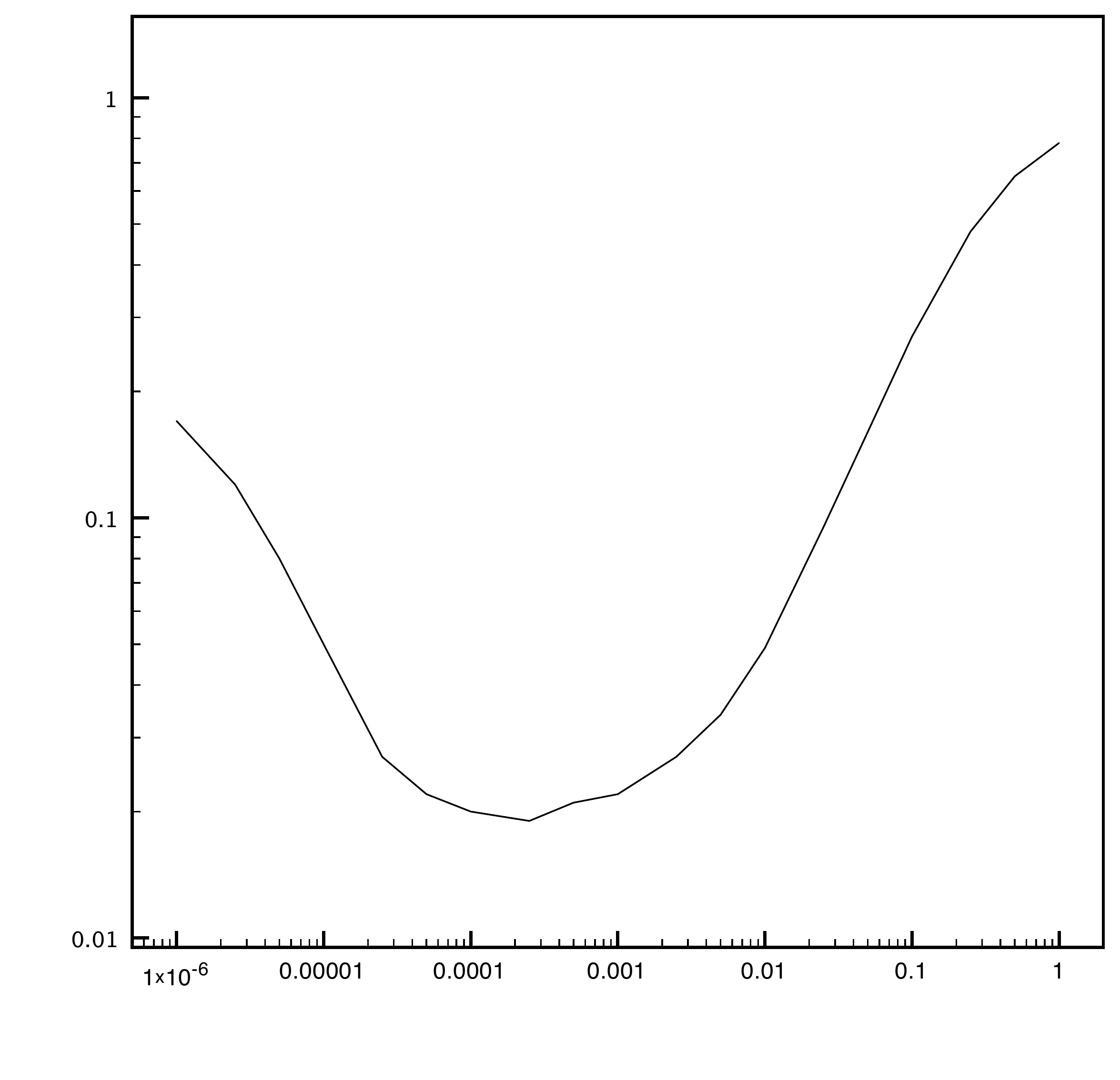}
\caption{Study of the global relative $L^2$-error  against penalty parameters
  $\gamma_V$ and $\gamma_W$ for the methods using
  \eqref{eq:weak_dual_stab_bc} (left) and \eqref{eq:num_stab1} (right).}\label{fig:parameter_study}
\end{figure}
\subsection{Convergence studies}
We know that the problem becomes increasingly ill-posed as $n$ becomes
large, but that the stabilities given by \eqref{cont_dep_assump} and
\eqref{stab_1}, \eqref{stab_2} hold independently of $n$. We
performed computations varying $n$ from $1$ to $5$ on a series of
unstructured meshes with approximate meshsizes in the
set, $$\{0.1,\,0.05,\,0.025, \, 0.0125, \, 0.008333\}.$$ Herein we only
present the results of the computations for odd $n$. The results are
given in Figures \ref{fig:n=1} - \ref{fig:n=5}. We have studied the
relative $L^2$-norms for the four different values of $\zeta$ given in
\eqref{relative_L2}. Each value of $\zeta$ is represented by a
different symbol according to $\zeta=1$, symbol: {\Large{$\circ$}};
$\zeta=1/2$, symbol: {\small{$\square$}};
$\zeta=1/4$, symbol:  {\Large{$\diamond$}}; $\zeta=1/8$, symbol:
{\small{$\bigtriangleup$}}. Filled symbols are used for graphs
representing the $H^1$-error. As we increase the value of $n$ the
$H^1$-norm of the exact solution, denoted $E$, increases and is given
in the captions of the figures. We see that the error level increases
with increasing $E$.
For the lower values $n=1$ and $n=3$ we observe
typically $O(h^{\frac32})$ convergence in the $L^2$-norm for all
quantities. The global error takes relatively smaller values for
higher $n$ compared to the local error quantities as an effect of the normalization. For $n=5$ the method
using \eqref{eq:weak_dual_stab_bc} appears to have approximately
$O(h)$ convergence. The same global convergence behavior is observed
for the method using \eqref{eq:num_stab1}, but in this case the
convergence is uneven although the errors are smaller than for
\eqref{eq:weak_dual_stab_bc}. The observed superconvergence compared
to the theoretical results can be attributed to the fact that in these
computations, the error in the $H^1$-norm also decreased, making the
constants $C(E)$ and $C_1(E)$ of equations \eqref{stab_1}, \eqref{stab_2} decrease as
well. We illustrate this for the case $n=5$ in Figure \ref{fig:H1n=5}.

 Observe that \eqref{stab_1} and
\eqref{stab_2} are valid also in the limit of $n \rightarrow \infty$
and it appears that the logarithmic continuous dependence is not
dominating on the relatively low values of $n$ and large values of
$h$, considered herein. For larger values of $n$ the $H^2$-norm of the
exact solution becomes so large that the computations on the meshes
considered are not in the asymptotic range. For $n=7$ the
linear decrease predicted in Proposition \ref{prop:stab_conv},
independently of the stability of the problem, was not observed.
\begin{figure}
\includegraphics[height=6cm]{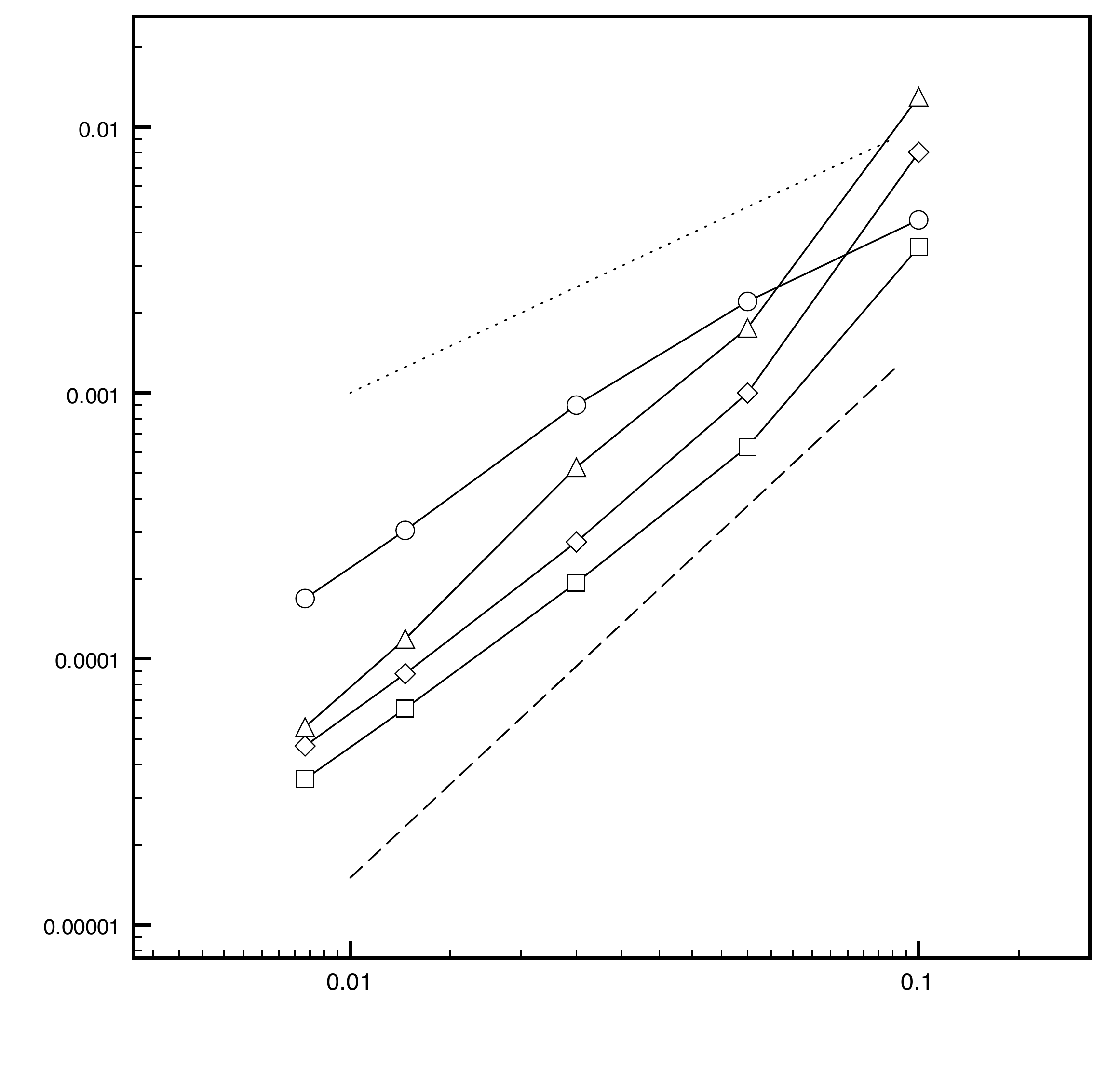}\includegraphics[height=6cm]{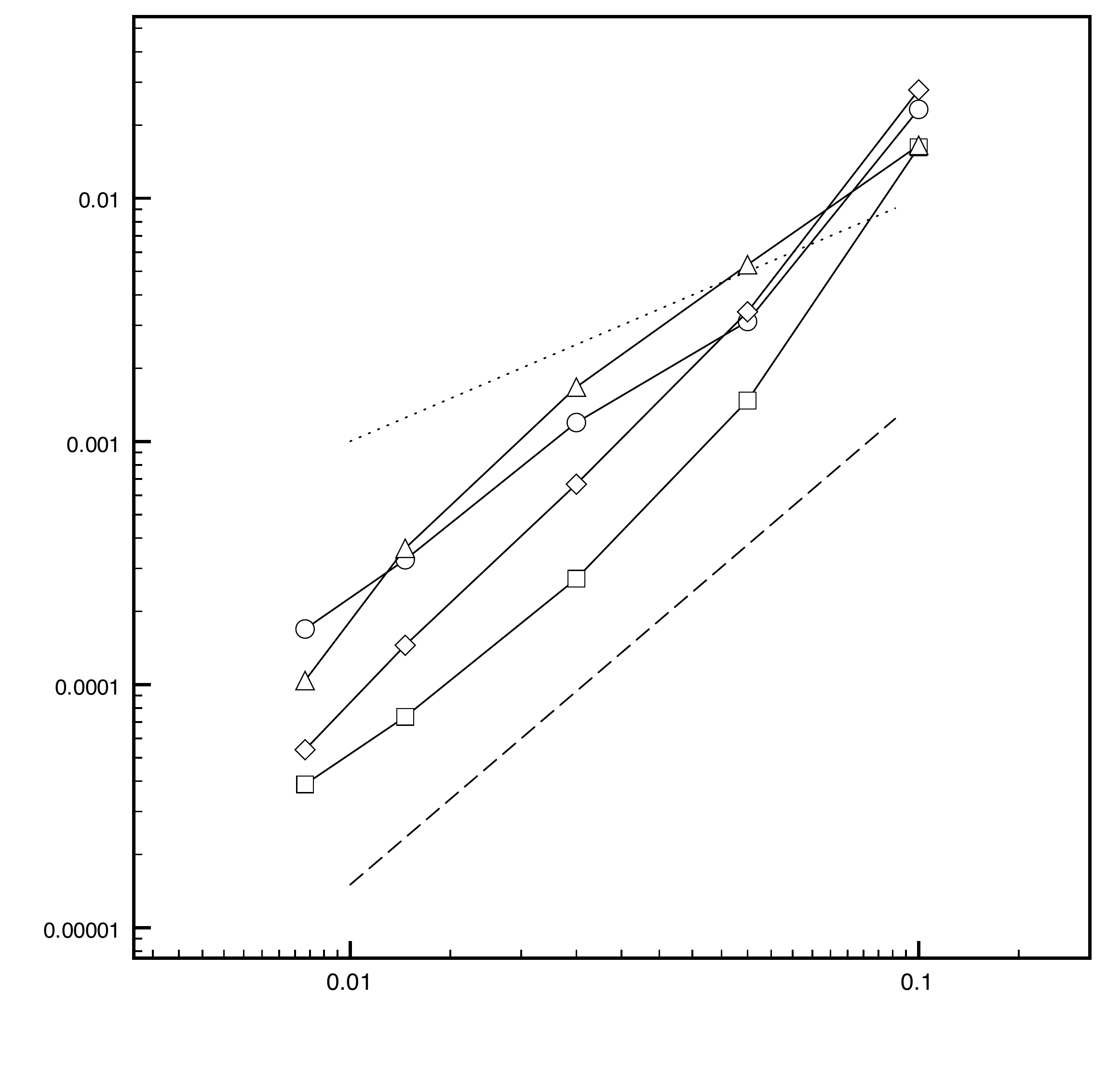}
\caption{Relative $L^2$-error  against mesh-size, $n=1$ in
  \eqref{Cauchy_data} and $E=1.68$, using the stabilizations
  \eqref{eq:weak_dual_stab_bc} (left) and \eqref{eq:num_stab1}
  (right).  Reference curves: $y = 0.1 x$ (dotted) and $y= 0.15 x^2$ (dashed)}\label{fig:n=1}
\end{figure}
\begin{figure}
\includegraphics[height=6cm]{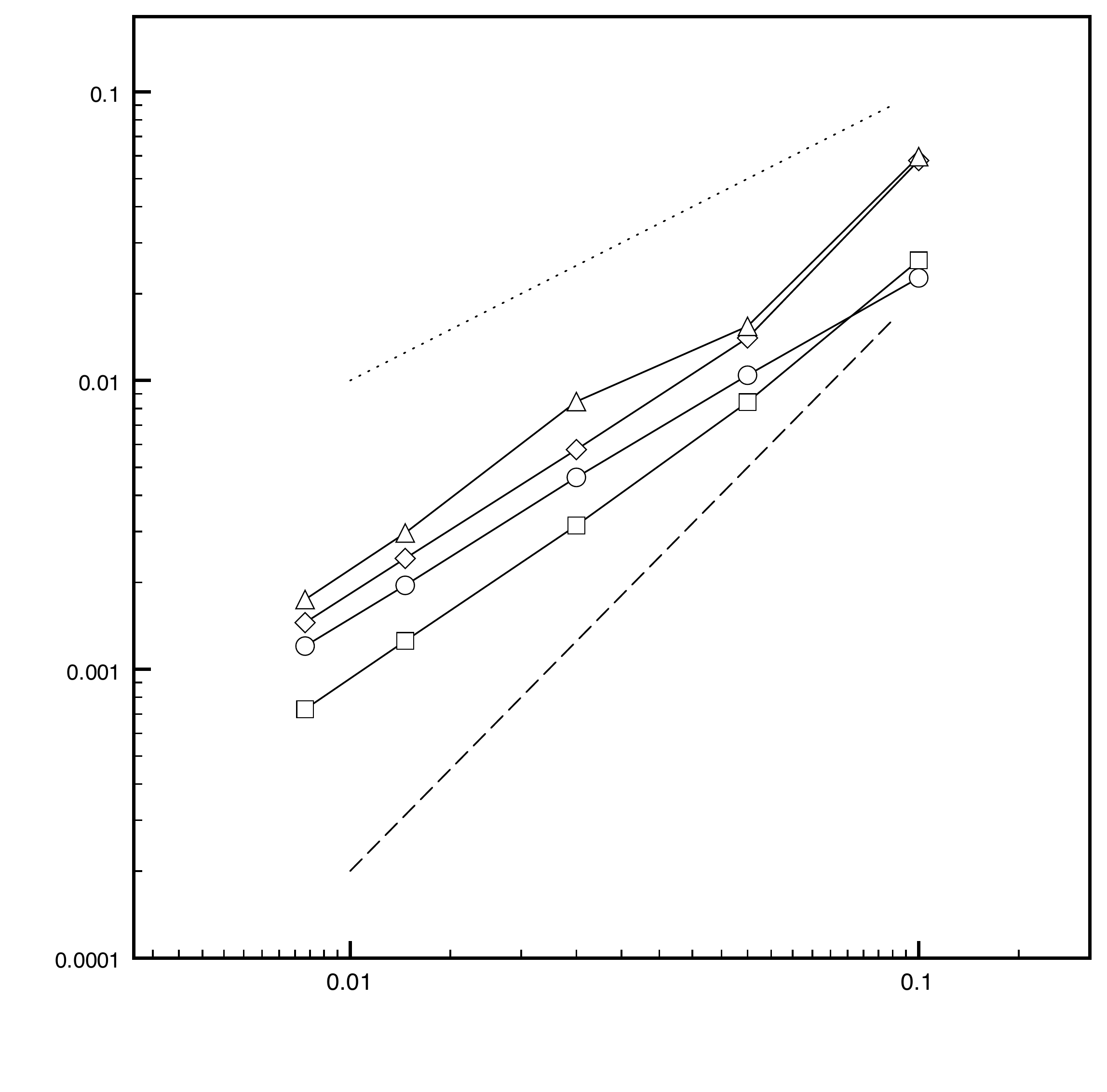}\includegraphics[height=6cm]{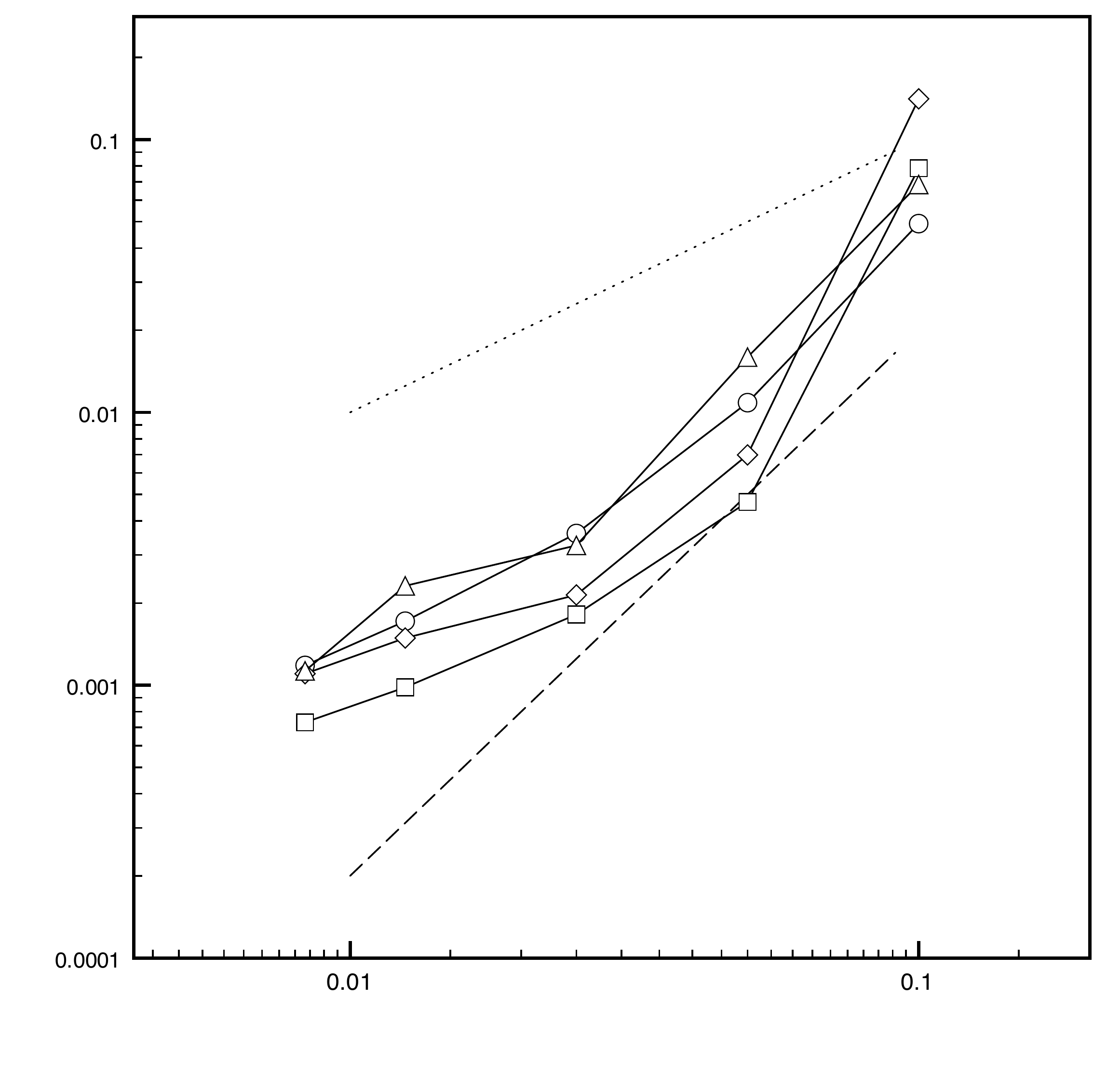}
\caption{Relative $L^2$-error  against mesh-size, $n=3$ in
  \eqref{Cauchy_data} and $E=7.27$, using the stabilizations
  \eqref{eq:weak_dual_stab_bc} (left) and \eqref{eq:num_stab1}
  (right). Reference curves: $y = x$ (dotted) and $y= 2x^2$ (dashed)}\label{fig:n=3}
\end{figure}
\begin{figure}
\includegraphics[height=6cm]{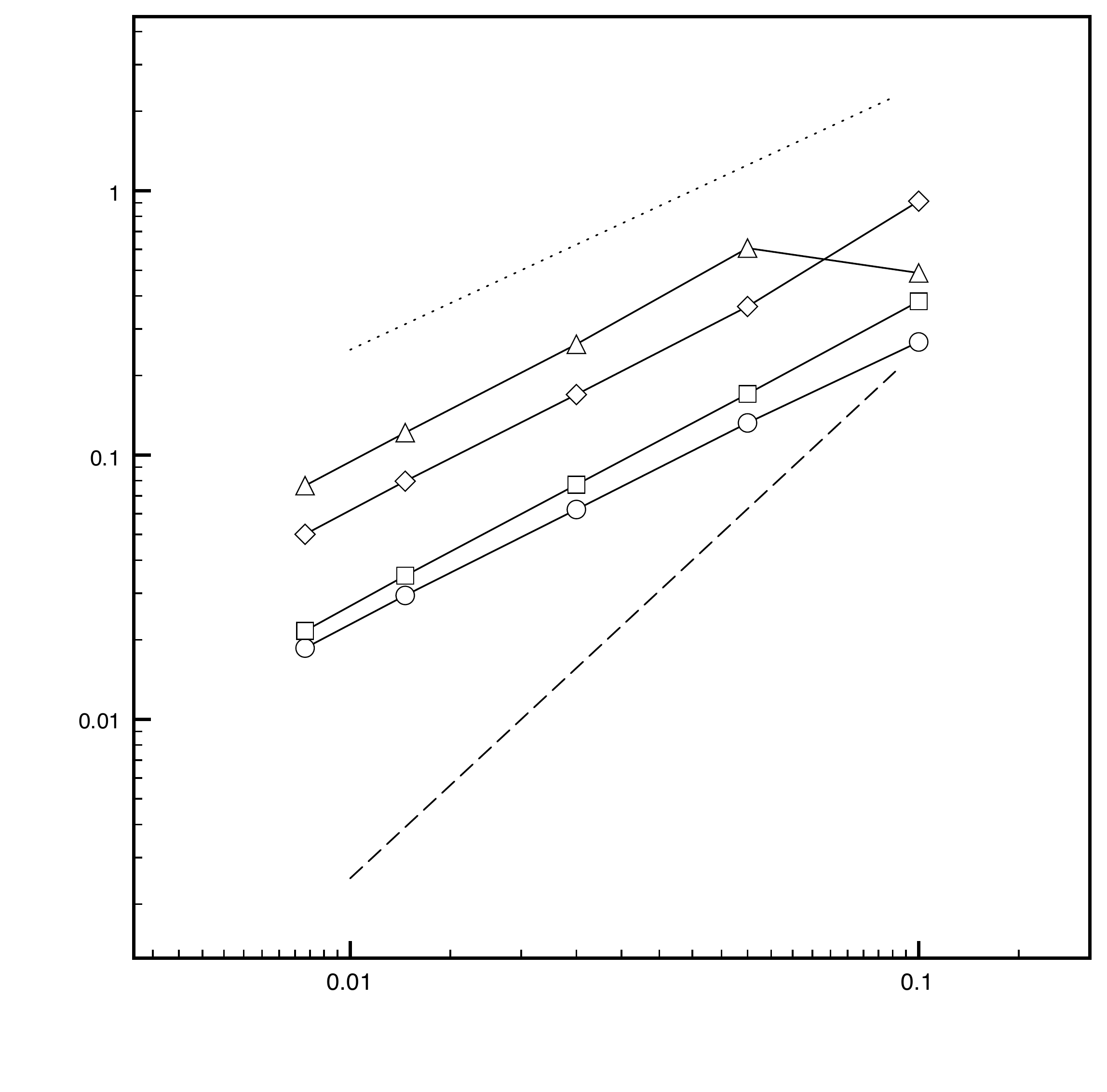}\includegraphics[height=6cm]{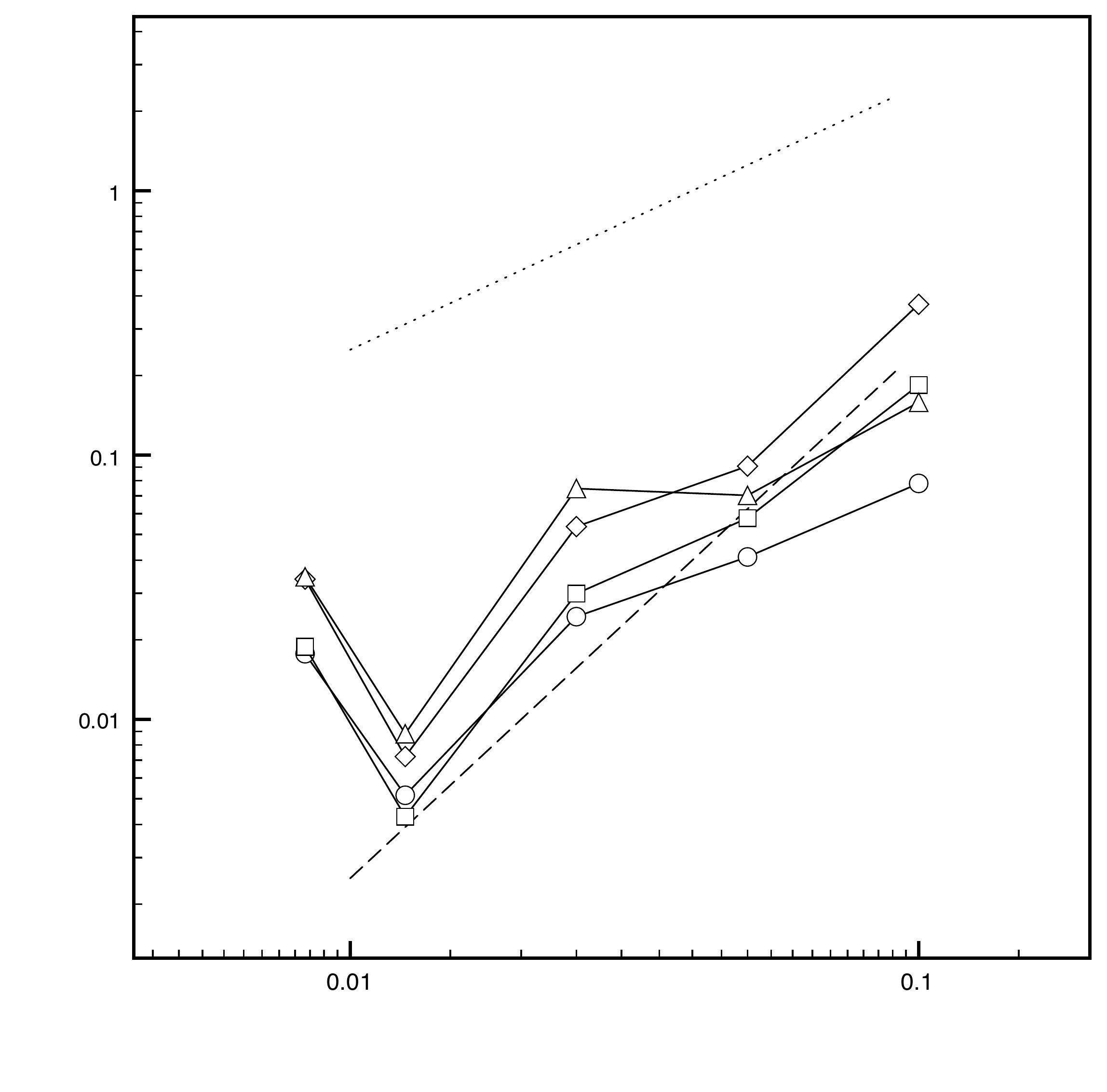}
\caption{Relative $L^2$-error  against mesh-size, $n=5$ in
  \eqref{Cauchy_data} and $E=41.6$, using the stabilizations
  \eqref{eq:weak_dual_stab_bc} (left) and \eqref{eq:num_stab1}
  (right).
 Reference curves: $y = 25 x$ (dotted) and $y= 25 x^2$ (dashed)}\label{fig:n=5}
\end{figure}
\begin{figure}
\includegraphics[height=6cm]{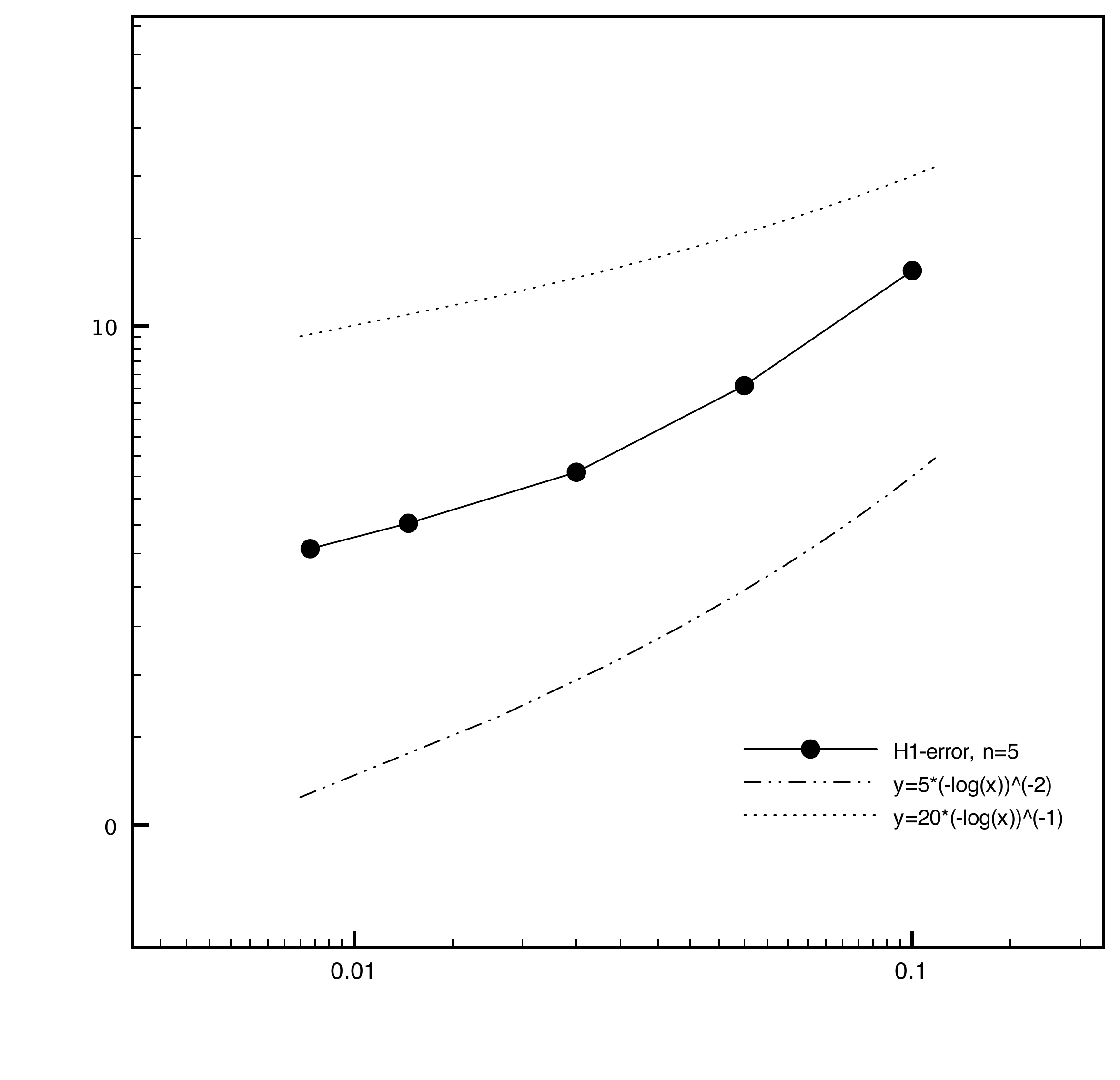}\includegraphics[height=6cm]{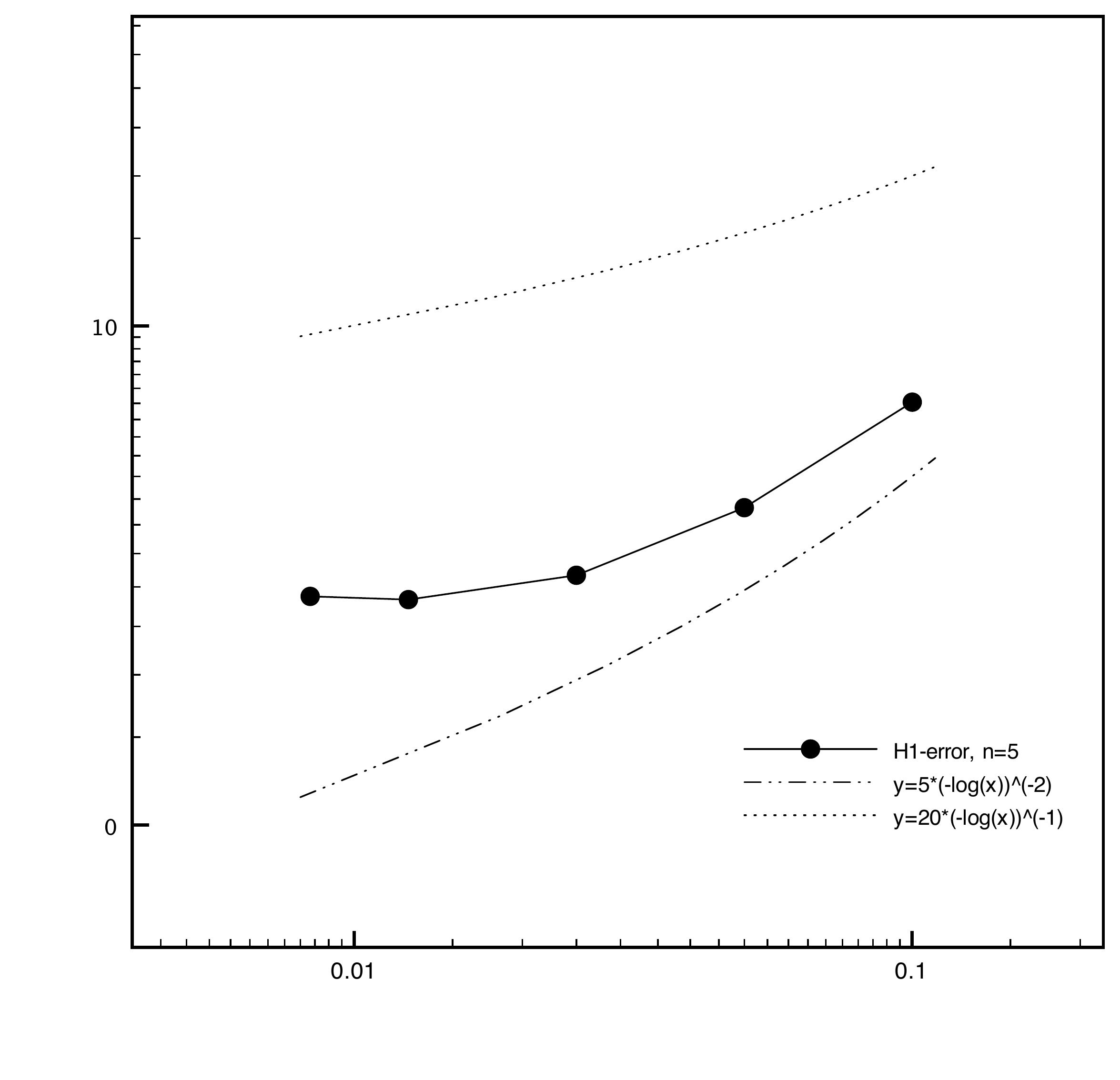}
\caption{Absolute $H^1$-error  against mesh-size, $n=5$ in
  \eqref{Cauchy_data} and $E=41.6$, using the stabilizations
  \eqref{eq:weak_dual_stab_bc} (left) and \eqref{eq:num_stab1}
  (right). Reference curves: $y = -20 log(x)^{-1}$ (dotted) and $y=
-5log(x)^{-2}$ (dash-dot) }\label{fig:H1n=5}
\end{figure}
\section{Concluding remarks}
We have proposed a nonconforming stabilized finite element method for
the approximation of elliptic Cauchy problems. Two different
stabilization operators were studied. The operator \eqref{eq:num_stab1} was shown
to give better control over perturbations in data, whereas \eqref{eq:weak_dual_stab_bc} is
adjoint consistent, possibly performing better for the computation of
certain linear functionals. We proved a posteriori and a priori error
estimates for both approaches under the assumption of continuous
dependence.
Numerically both methods were shown to have similar performance, but the method using
\eqref{eq:num_stab1} was sensitive to over-stabilization on high resultion
computations for high frequency solutions. This method also needed
separate tuning of the parameters $\gamma_V$ and $\gamma_W$, whereas
they could be chosen equal for the method using \eqref{eq:weak_dual_stab_bc}.
%    Bibliographies can be prepared with BibTeX using amsplain,
%    amsalpha, or (for "historical" overviews) natbib style.
\bibliographystyle{amsplain}
\bibliography{nc_Cauchy}

%    Insert the bibliography data here.

\end{document}